\documentclass[a4paper,10pt,notitlepage]{article}
\usepackage[affil-it]{authblk}
\usepackage[latin1]{inputenc}
\usepackage[T1]{fontenc}
\usepackage[english]{babel}
\usepackage{amsmath,amsthm}
\usepackage{amsfonts}
\usepackage{subfigure}
\usepackage{amssymb}
\usepackage{graphicx}
\usepackage{bbm}
\usepackage{color}
\usepackage{wasysym}
\usepackage[colorlinks=true, citecolor=red, linkcolor=blue, urlcolor=black]{hyperref}

\usepackage{tikz}
\usetikzlibrary{calc}
\usepackage{tikz,pgfplots}
    \usetikzlibrary{pgfplots.polar,intersections}
    \pgfplotsset{compat=newest}
\usetikzlibrary{hobby,patterns}

\newtheorem{theorem}{Theorem}[section]
\newtheorem{proposition}{Proposition}[section]
\newtheorem{lemma}{Lemma}[section]

\numberwithin{equation}{section}
\numberwithin{figure}{section}

\newcommand{\ve}[1]{{\boldsymbol {#1}}}
\newcommand{\R}{\mathbb R}

\providecommand{\keywords}[1]
{
  \small	
  \textbf{\textit{Keywords---}} #1
}
\usepackage{csquotes}

\title
{Random matching in 2D with exponent 2 for gaussian densities}

\author{ Emanuele Caglioti$\,\,^*$,  Francesca Pieroni
  \thanks{Electronic address: \texttt{caglioti@mat.uniroma1.it, francesca.pieroni@uniroma1.it}}}
\affil{Dipartimento di Matematica \\Sapienza Universit\`a di Roma}

\begin{document}
  \maketitle

\begin{abstract} 
We solve the Random Euclidean Matching problem with exponent 2 for the Gaussian distribution defined on the plane. 
Previous works by Ledoux and Talagrand determined the leading behavior of the average cost up to a multiplicative constant. 
We explicitly determine the constant, showing that the average cost is proportional to $(\log\, N)^2,$ where $N$ is the number of points. Our approach relies on a geometric decomposition allowing an explicit computation of the constant. Our results illustrate the potential for exact solutions of random matching problems for many distributions defined on unbounded domains on the plane.
 \end{abstract}

\section*{\small{Acknowledgement}}
This work has been partially supported by the grant \enquote{Progetti di ricerca di Ateneo
2019} by Sapienza University, Rome and  by PNRR MUR project PE0000013-FAIR.
\newpage
\tableofcontents

\newpage
  \section{Introduction}
\label{sez:intro}
Here we consider the Random Euclidean Matching problem with exponent 2 for some unbounded distribution in the plane.
The case of unbounded distributions is important both from a mathematical and an applied point of view.

In particular we consider the case of the Gaussian, the case of the Maxwellian and we give some hints on the case of other unbounded exponentially decaying densities.

\vskip.3truecm
Random Euclidean Matching is a combinatorial optimization problem in which $N$ red points and $N$ blue points extracted independently from the same probability distribution are paired, a red point with a single blue point and vice versa, in order to minimize the sum of the distances (the distance raised to a certain power) between the points.
The problem is equivalent to calculating the Wasserstein distance between the two empirical measures associated with the points or, which is the same, finding the optimal transport between the two empirical measures.
\vskip.3truecm
Apart for the great mathematical interest in the subject, in the last years applications of Matching and Optimal Transport in statistical learning and more in general in applied mathematics has enormously increased.
In particular we refer to  \cite{CP},\cite{BLG},\cite{FG} and references therein for applications of Optimal Transport to statistical learning, 
  computer graphics,  image processing, shape analysis, pattern recognition, particle systems and more.
For general reviews on Optimal Transport we refer to \cite{V} and with a more applicative aim to \cite{Sa}, while for a general review of results and tools and methods in probability, with applications to the matching problem we refer the reader to \cite {Treview}.
\vskip.3truecm
Coming to the mathematical problem, recently there has been a lot of activity on this topic because of some sharp progress, starting from a conjecture by Caracciolo et al., which shows how the problem can be rephrased in terms of PDE. 
\vskip.3truecm
Let $\mu$ be a probability distribution defined on $\Lambda\subset \R^2.$
Let us consider two sets $X^N
=\{X_i\}_{i=1}^N$ and 
$Y^N = \{Y_i\}_{i=1}^N$
of $N$ points independently sampled from the distribution $\mu$.
The Euclidean Matching problem with exponent $2$ consists in finding
the matching $i\to \pi_i$, i.e. the permutation $\pi$
of $\{1,\dots N\}$ which minimizes
the sum of the squares of the distances between $X_i$ and
$Y_{\pi_i}$,
that is
\begin{equation}
  \label{eq:cn}
  C_N(X^N,Y^N):= \min_{\pi}\sum_{i=1}^N
  |X_i-Y_{\pi_i}|^2.
\end{equation}
The cost defined above can be seen, but for a
constant factor $N,$ as the square of the 2-Wasserstein distance
between two probability measures.
In fact, the $p-$Wasserstein distance $W_p(\mu,\nu)$,
with exponent $p\geq1$, between two probability
measures $\mu$ and  $\nu,$ is defined by 
\begin{equation*}
  W_p^p(\mu,\nu):=\inf_{J_{\mu,\nu}}
  \int dJ_{\mu, \nu}
  (x,y) |x- y|^p,
\end{equation*}
where the infimum is taken on all the joint probability
distributions $dJ_{\mu,\nu}(x,y)$
with marginals with respect to $dx$ and $dy$ given by
$\mu$ and $\nu$, respectively.  
Defining the empirical measures
$$
\mu^N:=\frac{1}{N}\sum_{i=1}^N\delta_{X_i},\phantom{aa}
\nu^N:=\frac{1}{N}\sum_{i=1}^N\delta_{Y_i},
$$
it is possible to show that
$$C_N(X^N, Y^N) = N W^2_2(\mu^N,\nu^N),$$
(see for instance \cite{B}).
In the sequel we will shorten
$C_N: = C_N(X^N,Y^N)$.

The first general result on Random Euclidean Matching was obtained by combinatorial arguments in \cite{AKT}. 
In particular, in the case of dimension 2 and exponent 2, assumed that $X_i$ and $Y_i$ are
independently sampled with uniform density on the unit square $Q$ they prove that
$\mathbb{E}_\sigma[W_2^2(\mu^N,\nu^N)] $
behaves like $\frac{\log N}{N}$, where with $\mathbb {E}_\sigma$ we have indicated with the expected value
with respect to the uniform distribution $d\sigma(x) = dx$
of the points $\{X_i\}$ and $\{Y_i\}$.

In the challenging paper \cite{CLPS}, Caracciolo et al. conjecture that
\begin{equation}
    \label{eq:ecn}
    \mathbb {E}_\sigma[C_N] \sim \frac{\log N}{2\pi},
\end{equation}
where we say that $f\sim g$
if $\lim_{N\rightarrow+\infty} f(N)/g(N) =1$.
In terms of $W_2^2$ the conjecture is equivalent to
\begin{equation}
    \label{done1}
    \mathbb{E}_\sigma[W_2^2(\mu^N,\nu^N)] \sim \frac{\log N}{2\pi N}.
\end{equation}

Furthermore, in \cite{CLPS} it is conjectured that asymptotically the
expected value of
$W_2^2(\mu^N,\sigma)$ between the empirical density $X^N$ and the uniform probability measure $\sigma$ on $Q$ is given by
\begin{equation}
    \label{done2}
    \mathbb{E}_\sigma[W_2^2(\mu^N,\sigma)]
    \sim \frac{\log N}{4\pi N}.
\end{equation}

The above conjectures were proved by Ambrosio et al. \cite{AST}.
In \cite{AG} more precise estimates are given and it is proved
that the result can be extended to the case where
particles are sampled from the uniform measure on a two-dimensional
Riemannian compact manifold. In \cite{AGlT} it is shown that the
optimal transport map
for $W_2(\mu^N,\sigma)$ can be approximated as conjectured in \cite{CLPS}.

We notice that, by simple scaling arguments, if we consider squares or manifold $\Lambda$ of measure $|\Lambda|$,
the cost has to be multiplied by $|\Lambda|$.
\vskip.3cm

Then, in \cite{BC} it has been conjectured that, if the points are sampled from a smooth and strictly positive density $\rho$ in a regular set $\Lambda,$ then the result is the same: i.e.
  the leading term of the expected value of the cost is $ \frac{|\Lambda|}{2\pi}\log N.$
  The conjecture is based on a linearization of Monge-Ampere equation close to a non uniform density and  a proof of the estimate form above is given when $\Lambda$ is a square.   
  This result has been proved by Ambrosio et al. \cite{AGT}. In particular they generalize the result to H\"{o}lder continuous positive densities in bounded regular sets and in Riemannian manifolds.
  
  Summarizing, if the density $\sigma=\rho dx ,$ is supported in a bounded regular set $\Lambda$ where $\rho$ is Holder continuous, and if there exist constants $a$ and $b$ such that $0 < a <\rho<b,$ then 
\begin{eqnarray}\label{univ}
\mathbb {E}_\sigma[C_N] \sim \frac{|\Lambda|}{2\pi}\log N
\end{eqnarray}
  
In \cite{BCCDSS}, in the case of constant densities, the correction to the leading behavior has been studied. In particular it is conjectured that the correction is given in terms of the regularized trace of the inverse of the Laplace operator in the set.   
  
\subsection{Main results}
Interestingly \eqref{univ} implies that the limiting average cost is not continuous in the space of densities, even in $L_{\infty}$ norm.
 
Indeed, if we consider a sequence of smooth strictly positive densities $\rho_k$ on a disk of radius $2,$  converging, as $k\rightarrow\infty$ to $\rho=\frac{1}{\pi}1_{|x|<1},$ that is the uniform density on the disk of radius $1,$  we get that for any $k:$ $\mathbb {E}_\sigma[C_N] \sim 2\log N,$ while for the limiting density $\rho$ we get 
 $\mathbb {E}_\sigma[C_N] \sim \frac12\log N.$
 
 It is therefore natural to ask if it is possible to define sequences of densities, positive  on all the disk of radius $2,$ that converge to the density $\rho=\frac{1}{\pi}1_{|x|<1},$ and such that $\mathbb {E}_\sigma[C_N] \sim c\log N,$ where $c\in(\frac12,2).$ The answer is yes.

For instance, if we consider, in the disk of radius $2$ the sequence of $N$-dependent "multiscaling" densities
\begin{equation}
	\rho_N=
	\begin{cases}
	\frac{1}{\pi}\left(1-N^{\alpha-1}\right)\,; 0<|x|\leq 1\\
	\frac1{3\pi} N^{\alpha-1}\phantom{,,,}; 1 < |x|\leq2
	\end{cases}
\end{equation}
where $0< \alpha<1.$ That is, in the average, there are $N-N^\alpha$ points in the disk and $N^\alpha$ points in the annulus. 
 
 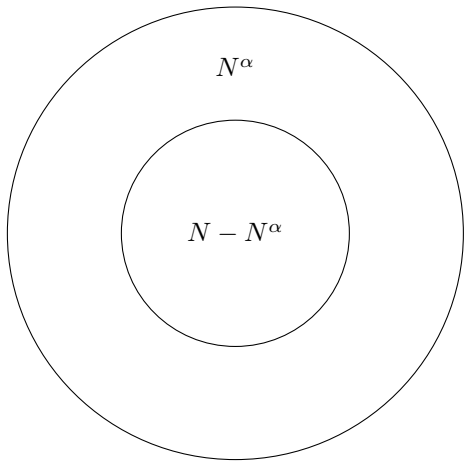
\begin{figure}[hbt!]
\centering
 \begin{tikzpicture}
 \node[] at (2,2) {$N-N^\alpha$};
 \node[] at (2,4.2) {$N^\alpha$};
 \draw (2,2) circle (1.5cm);
\draw (2,2) circle (3cm);
\end{tikzpicture}
\caption{Multiscaling density} \label{fig:M1}
 \end{figure}
 In this case if $d\sigma(x)=\rho_N(x)dx$, as we prove in Section \ref{sec2}, the average of the cost is given by 
 $$\mathbb {E}_\sigma[C_N] \sim\frac{1}{2\pi}(\pi \log N + 3\pi\log N^\alpha)=\left(\frac12+\frac32 \alpha\right)\log N$$
Here we consider three problems.

The first is a generalization of the example seen above to any finite number of circular annuli.
 As we shall see this can be considered as a toy model for the Gaussian case.
 
Under suitable monotonicity conditions we will prove, see Theorem \ref{multiscalingthm}, that the average cost is 
given by
$$\mathbb{E}_\sigma[C_N]\sim\frac{1}{2\pi}\int_{\Lambda}dx \log (N \rho_N(x))$$
 
 The second case is the case of the Gaussian distribution, that is $d\mu(x)=\rho(x)dx$ and
 $$\rho(x)=\frac{1}{2\pi}e^{-\frac{|x|^2}{2}}.$$ 
 In this case 
Talagrand proved
  \cite{Ta}, that the average cost, for large $N$
  satisfies
    $$ \frac1C(\log N)^2 \leq \mathbb{E}_\mu[C_N]\leq C  (\log N)^2.$$
  An estimate from above proportional to $(\log N)^2$ was
  previously proved by Ledoux in \cite{Le}, see also \cite{Le2} where an estimate from below is proved using PDE techniques as in \cite{AST}.
  
 In this case we prove, see Theorem \ref{gauxthm}, that the average cost is 
 $$\mathbb{E}_\mu[C_N]\sim\frac12  (\log N)^2 .$$ 
 The third case, is when the density is given by 
 $$\rho(x)=\chi_{[0,1]}(x_1)\frac1{\sqrt{2\pi}}e^{-\frac{x_2^2}2}.$$
and again $d\mu(x)=\rho(x)dx$. This density, interpreting $x_2$ as a velocity, is simply the Maxwellian distribution for a gas in the box (the segment) $[0,1].$  
 
 In this case we prove, see Theorem \ref{Maxthm}, that the average limit cost is 
 $$\mathbb{E}_\mu[C_N]\sim\frac{2\sqrt{2}}{3\pi} (\log N)^{3/2} .$$ 
 As we shall see the problem of the Gaussian and the problem of the Maxwellian can be considered as the limit of a suitable sequence of the multiscaling densities, in particular, in both cases we obtain that \begin{equation}\mathbb{E}_\mu[C_N]\sim\frac{1}{2\pi}\int dx[\log (N \rho(x))]_+ \label{generalbehavior}\end{equation}
 Dealing with other radially symmetric exponentilaly decaying densities, that is with densities proportional to $e^{-|x|^\alpha}$, Talagrand showed that the leading behavior is  proportional to $(\log N)^{1+2/\alpha}.$
 We think that it would be possible to modify the proofs given here to deal with these cases and to get also in this case the exact leading behavior (i.e. determining also the multiplicative constant). More precisely, by  \eqref{generalbehavior} we would get 
 $$\mathbb{E}_\mu[C_N]\sim \frac{\alpha}{4+2\alpha}(\log N)^{1+2/\alpha} $$
 Note that for $\alpha = 2$ we find the constant $\frac14$ instead of $\frac12$  because here we are considering the Gaussian $e^{-|x|^2}$  and not $e^{-|x|^2/2}.$
 
 To prove this is out of the aim of the present paper and the technique would be slightly different, because here we make use of the fact that the gaussian is a product measure.
 
For what concerns probability distributions that decays as a power of the distance, for instance $\frac{1}{1+|x|^\alpha},$ we do not dare to make conjectures. In that case the slow decay of the distribution does not allow us to apply techniques similar to those used in this work.
 
 \vskip.3cm
 The structure of the paper is the following.
 
 In Section \ref{subsec1} we give some general results on Wasserstein distance that we use in the sequel.
 
In Section \ref{sec2} we consider the case of multiscaling densities, in Section \ref{sec3} the case of the Gaussian density and in Section \ref{sec1} the case of the Maxwellian density.

\vskip.3cm

 Now we briefly review what it is known, up to our knowledge, on the Random Euclidean Matching in dimension $d\neq 2,$ with particular attention to the case of the constant distribution in the unite cube and of the Gaussian distribution.

In dimension $1$ the Random Euclidean Matching problem  is
almost completely characterized, for any $p\geq 1.$ This is due to the
fact that the best matching between two set of points on a line is 
monotone, see for instance \cite{CS2014}, \cite{CDS}, and \cite{BL} where 
a general discussion on the one-dimensional case, also for the case of
non-constant densities is given.  
In particular, for a segment of lenght $1$ and for $p=2:$ $\mathbb{E}[C_N]\to 1/3$ as $N\to\infty.$ For the normal distribution in dimension $1$ 
 in \cite{BF} it is proved that $\mathbb{E}[C_N]\sim \log \log N,$ while in \cite{BL} estimates from below and from above proportional to $\log \log N$ were given.
 
 In dimension $d\geq 3,$ for the constant density in a cube, it has been proved that 
$\mathbb{E}[C_N]$ behaves as  $N^{1-p/d},$  for any $p\geq 1$ (see \cite{T1996},
\cite{DY}, \cite{Le}). In  \cite{GT} it has been proved the existence of the limit $\frac{\mathbb{E}[C_N]}{N^{1-p/d}}$ for any $p\geq 1$).
 
In dimension $d\geq 3$, the case of unbounded densities and in particular the gaussian case has been widely studied, see \cite{Ta},
\cite{DY},   \cite{BaB}, \cite{Le3}.
In particular, in \cite{BaB}, it has been proved that $\mathbb{E}[C_N]$ behaves as $N^{1-p/d},$ for any $0 < p < d/2,$ and an explicit expression for the constant multiplying $N^{1-p/d}$ is conjectured, while
in \cite{Le3}, it has been proved that 
$\mathbb{E}[C_N]$ behaves as $N^{1-p/d},$   for any $1\leq p < d.$
General results on Random Euclidean Matching, including the case $p > d$ and the case of unbounded densities are given in  \cite{FouG}.

\section{Useful results and notations}\label{subsec1}

In this Section, we recall some preliminary results that we will need later.

The following Lemma links the cost of semidiscrete problem to the cost of bipartite one.
\begin{lemma}[\cite{AST}, Proposition 4.8]\label{semibip}
Let $\rho$ be any probability density on $\mathbb{R}^2$ and let $X_1,\dots,X_N$ and $Y_1,\dots,Y_N$ independent random variables in $\mathbb{R}^2$ with common distribution $\rho$, then
$$
\mathbb{E}\left[W_2^2\left(\frac{1}{N}\sum_{i=1}^N\delta_{X_i},\frac{1}{N}\sum_{i=1}^N\delta_{Y_i}\right)\right]\leq2\mathbb{E}\left[W_2^2\left(\frac{1}{N}\sum_{i=1}^N\delta_{X_i},\rho\right)\right]
$$
\end{lemma}
We will use also the following property for the upper bounds of the leading terms, that is a consequence Benamou-Brenier formula.
\begin{theorem}[\cite{BB}, Benamou-Brenier formula]
If $\mu$ and $\nu$ are probability measures on $\Omega\subseteq\mathbb{R}^d$, then 
\begin{eqnarray*}
W_2^2(\mu,\nu)=\inf\left\{\left.\int_0^1dt\int_{\Omega}d\mu_t|v_t|^2\right|\begin{array}{lcr}\frac{d}{dt}\mu_t+{\it div}(v_t\mu_t)=0\\\mu_0=\mu,\mu_1=\nu\end{array}\right\}
\end{eqnarray*}
\end{theorem}
The result we will use is the following.
\begin{theorem}[\cite{AG}, Corollary 4.4]\label{benbrefor}
If $\mu$ and $\nu$ are probability densities on $\Omega\subseteq\mathbb{R}^d$ and $\phi$ is a weak solution of $\Delta\phi=\mu-\nu$ with Neumann boundary conditions, then
\begin{eqnarray*}
W_2^2(\mu,\nu)\leq4\int_{\Omega}dx\frac{|\nabla\phi(x)|^2}{\mu(x)}
\end{eqnarray*}
\end{theorem}
We will also use a result by Talagrand, that relates the Wasserstein distance with the relative entropy, that is the following.
\begin{theorem}[\cite{T1996}, Talagrand formula]\label{talfor}
Let $\rho$ be the Gaussian density in $\mathbb{R}^d$, that is $\rho(x)=\frac{1}{\sqrt{2\pi}^d}e^{-\frac{|x|^2}{2}}$. If $\mu$ is another density on $\mathbb{R}^d$, then we have
\begin{eqnarray*}
W_2^2(\rho,\mu)\leq\int_{\mathbb{R}^d}dx\mu(x)\log\frac{\mu(x)}{\rho(x)}
\end{eqnarray*}
\end{theorem}
Then, when proving the convergence, while for the upper bound we will use the canonical Wasserstein distance, for the lower bound we will use, as in \cite{AGT}, a distance between non-negative measures introduced in \cite{FG}, that is
\begin{eqnarray}\label{figalligigli}
Wb_2^2(\mu,\nu):=\inf\left\{\left.\int_{\bar\Omega\times\bar\Omega}dJ(x,y)|x-y|^2\right|\begin{array}{lc}\pi^1_{\#}J|_{\Omega}=\mu\\\pi^2_{\#}J|_{\Omega}=\nu\end{array}\right\}
\end{eqnarray}
In the following Theorem we denote with $W_{2,*}$ either the canonical $W_2$ and the boundary version of it, that is $Wb_2$, and we collect some known results that we will need later.

\begin{theorem}[\cite{AGT}, Theorems 1.1 \& 1.2, Propositions 3.1 \& 3.2]\label{teoremone}
Let $\Omega\subseteq\mathbb{R}^2$ be a bounded connected domain with Lipschitz boundary and let $\rho$ be a H\"older continuous probability density on $\Omega$ uniformly strictly positive and bounded from above. Given iid random variables $\{X_i\}_{i=1}^N$ and $\{Y_i\}_{i=1}^M$ with common distribution $\rho$, we have
\begin{eqnarray}
\label{teoremone1}&&\frac{N}{\log N}\mathbb{E}\left[W_{2,*}^2\left(\frac{1}{N}\sum_{i=1}^N\delta_{X_i},\rho\right)\right]\xrightarrow[N\to\infty]{}\frac{|\Omega|}{4\pi}
\\
\label{teoremone2}&&\frac{N}{\log N}\mathbb{E}\left[W_{2,*}^2\left(\frac{1}{N}\sum_{i=1}^N\delta_{X_i},\frac{1}{M}\sum_{i=1}^M\delta_{Y_i}\right)\right]\xrightarrow[N,M\to\infty]{\frac{N}{M}\to q}\frac{|\Omega|}{4\pi}(1+q)
\end{eqnarray}
\end{theorem}

Finally, we specify that hereafter we will denote the expected value conditioned to a random variable $X$ as $\mathbb{E}_{X}[\cdot]:=\mathbb{E}[\cdot|X]$.

In the sequel, we denote with the same symbol a probability measure absolutely continuous with respect to Lebesgue measure and its density.

\section{A piecewise multiscaling density}\label{sec2}

In this Section, we examine the trasportation cost of a random matching problem when $X_1,\dots,X_N$ and $Y_1,\dots,Y_N$ are independent random variables in the disk $C=\{|x|\leq S\}$ with common distribution $\rho^L_N$, defined by
\begin{equation}
\rho^L_N(x):=\frac{N-\sum_{l=1}^{L-1}N^{\alpha_l}}{N}\frac{\mathbbm{1}_{C_0}(x)}{|C_0|}+\sum_{l=1}^{L-1}\frac{N^{\alpha_l}}{N}\frac{\mathbbm{1}_{C_l}(x)}{|C_l|}
\label{distrcorone}
\end{equation}
where we have chosen the exponents $\alpha_l$ strictly positive and decreasing with the index $l$, $\alpha_0:=1,$ and where the annuli $C_l$ are defined by
$$
C_l:=\{s_l<|x|<s_{l+1}\}\quad;\quad0=s_0<s_1<\dots<s_{L-1}<s_L=S
$$
This density is piecewise constant on the annuli $C_l$, it depends on the number of particles we are considering and it allows to have (in the expected value) $N^{\alpha_l}$ particles (or $N-\sum_{l=1}^{L-1}N^{\alpha_l}$ if $l=0$) in the annulus $C_l$.
 \begin{figure}[hbt!]
\centering
 \begin{tikzpicture}
 \node[] at (2,2) {$N-N^{\alpha_1}- ...$};
 \node[] at (2,3.85) {$N^{\alpha_1}$};
 \node[] at (2,4.6) {...};
  \node[] at (2,5.3) {$N^{\alpha_{L-1}}$};
 \draw (2,2) circle (1.5cm);
\draw (2,2) circle (2.2cm);
\draw (2,2) circle (2.9cm);
\draw (2,2) circle (3.6cm);
\end{tikzpicture}
\caption{Multiscaling density} \label{fig:M7}
 \end{figure}
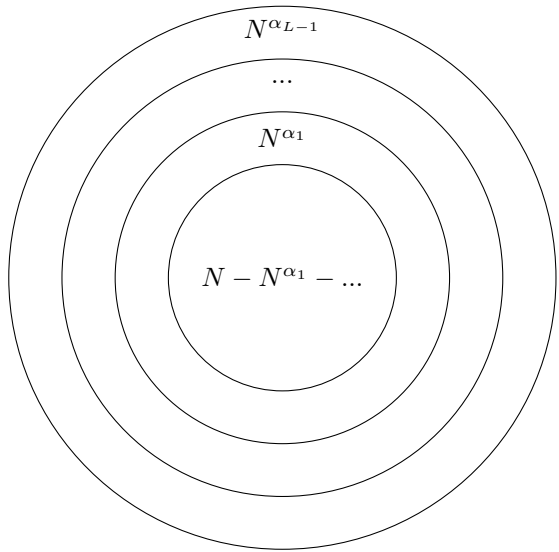
Here we prove the following theorem.
\begin{theorem}\label{multiscalingthm}
If $X_1,\dots,X_N$ and $Y_1,\dots,Y_N$ are iid random variables with common distribution $\rho^L_N$, it holds
\begin{eqnarray*}
&&\frac{N}{\log N}\mathbb{E}\left[W_2^2\left(\frac{1}{N}\sum_{i=1}^N\delta_{X_i},\rho^L_N\right)\right]\xrightarrow[N\to\infty]{}\frac{1}{4\pi}\sum_{l=0}^{L-1}\alpha_l|C_l|
\\
&&\frac{N}{\log N}\mathbb{E}\left[W_2^2\left(\frac{1}{N}\sum_{i=1}^N\delta_{X_i},\frac{1}{N}\sum_{i=1}^N\delta_{Y_i}\right)\right]\xrightarrow[N\to\infty]{}\frac{1}{2\pi}\sum_{l=0}^{L-1}\alpha_l|C_l|
\end{eqnarray*}
\end{theorem}
Let us notice that also if $\rho^L_N$ is supported on all the disk $C$ the asymptotic cost of the problem (except for a factor $2\pi$ or $4\pi$) is multiplied for $\sum_{l=0}^{L-1}\alpha_l|C_l|$, and
$$
\frac{\log N}{N}|C_0|=\frac{\log N}{N}\alpha_0|C_0|<\frac{\log N}{N}\sum_{l=0}^{L-1}\alpha_l|C_l|<\frac{\log N}{N}\sum_{l=0}^{L-1}|C_l|=|C|\frac{\log N}{N}
$$
therefore the cost is strictly smaller than the cost of the problem with particles distributed with a density bounded from below from a positive constant, as proved in \cite{AGT}. This happens because $\rho^L_N$ is not bounded from below: except for the disk $C_0$ it is everywhere vanishing for large $N$.

We can also notice that
$$
\rho^L_N\xrightarrow[N\to\infty]{\|\|_{\infty}}\frac{\mathbbm{1}_{C_0}}{|C_0|}=\mu_0
$$
while the total cost is strictly larger than the cost of the problem when the particles are distributed with measure $\mu_0$.

Finally, let us notice that the second statement of Theorem \ref{multiscalingthm} is equivalent to
 $$\mathbb{E}[C_N]\sim\frac{1}{2\pi}\int dx \log (N \rho^L_N(x))$$

First we prove the following Lemma, similar to propositions proved in \cite{BC} and \cite{AGT}. It allows to compute the total cost as the sum of the costs of the problems on the annuli. The argument used to estimate the Wasserstein distance between two measures that are not bounded from below is that when we use Benamou-Brenier formula we find a divergent term due to a vanishing denominator. This term in the annulus $C_l$ is balanced from the numerator, which involves the fluctuations of the particles in $C_l$ and in $\cup_{l=0}^{L-1}C_l$, whose order is the same thanks to the choice of the exponents $\alpha_l$.

\begin{lemma}\label{okannuli}
There exists a constant $c>0$ such that if $X_1,\dots,X_N$ and $Y_1,\dots,Y_N$ are independent random variables in $C$ with common distribution $\rho^L_N$, $N_l$ and $M_l$ are respectively the number of points $X_i$ and $Y_i$ in $C_l$, i.e. $N_l:=\sum_{i=1}^N\mathbbm{1}(X_i\in C_l)$ and $M_l:=\sum_{i=1}^N\mathbbm{1}(Y_i\in C_l)$, $\theta$ small enough and
$$
A_{\theta}:=\left\{\begin{array}{lcr}|N_l-N\rho^L_N(C_l)|&\leq&\theta N\rho^L_N(C_l)\\|M_l-N\rho^L_N(C_l)|&\leq&\theta N\rho^L_N(C_l)\end{array}\forall l=0,\dots,L\right\}
$$
it holds
\begin{eqnarray}
\label{okannuli5}&&\mathbb{E}\left[W_2^2\left(\rho^L_N,\sum_{l=0}^{L-1}\frac{N_l}{N}\frac{\rho^L_N\mathbbm{1}_{C_l}}{\rho^L_N(C_l)}\right)\right]\leq\frac{c}{N}
\\
\label{okannuli4}&&\mathbb{E}\left[W_2^2\left(\rho^L_N,\sum_{l=0}^{L-1}\frac{M_l-N_l}{3\theta N}+\frac{M_l}{N}\frac{\rho^L_N\mathbbm{1}_{C_l}}{\rho^L_N(C_l)}\right)\mathbbm{1}_{A_{\theta}}\right]\leq\frac{c}{\theta^2N}
\\
\label{okannuli6}&&\mathbb{E}\left[W_2^2\left(\sum_{l=0}^{L-1}\frac{3\theta M_l}{N}\frac{\rho^L_N\mathbbm{1}_{C_l}}{\rho^L_N(C_l)},\sum_{l=0}^{L-1}\frac{M_l-N_l+3\theta M_l}{N}\frac{\rho^L_N\mathbbm{1}_{C_l}}{\rho^L_N(C_l)}\right)\mathbbm{1}_{A_{\theta}}\right]\leq\frac{c}{\theta N}
\end{eqnarray}
\end{lemma}

\begin{proof} Since the proofs are very similar, we only focus on \eqref{okannuli5} and then we explain how to obtain \eqref{okannuli4} and \eqref{okannuli6}.

Let $f_N$ be the weak solution of
$$
\Delta f_N=\rho^L_N-\sum_{l=0}^{L-1}\frac{N_l}{N}\frac{\rho^L_N\mathbbm{1}_{C_l}}{\rho^L_N(C_l)}
$$
then by Theorem \ref{benbrefor} we get
\begin{eqnarray}
\nonumber W_2^2\left(\rho^L_N,\sum_{l=0}^{L-1}\frac{N_l}{N}\frac{\rho^L_N\mathbbm{1}_{C_l}}{\rho^L_N(C_l)}\right)&\leq&4\int_Cdx\frac{|\nabla f_N(x)|^2}{\rho^L_N(x)}
\\
\label{okannuli2}&=&4\sum_{l=0}^{L-1}\int_{C_l}dx|\nabla f_N(x)|^2\frac{1}{\rho^L_N|_{C_l}}
\end{eqnarray}
As explained before, now we are going to prove that, when we take the expectation, the divergent term due to the vanishing density is balanced by the small fluctuations of the particles.

We can find $f_N$ depending only on $|x|$, i.e. in the form
$$
f_N(x)=\int_0^{|x|}dr\frac{1}{r}\int_0^rdss\Delta f_N(s)+constant
$$
then if we define
$$
B_l:=\cup_{i=0}^{l-1}C_i\quad;\quad P_l:=\sum_{i=0}^{l-1}N_i=\sum_{i=1}^N\mathbbm{1}(X_i\in B_l)
$$
and we observe that the factor $2\pi s$ is exactly what we need to write the integral in polar coordinates, we have
\begin{eqnarray}
\nonumber\int_{C_l}dx|\nabla f_N(x)|^2&=&\int_{s_l}^{s_{l+1}}dr\frac{1}{2\pi r}\left(\int_{B_l\cup \{s_l<|x|<r\}}dx\Delta f_N(x)\right)^2
\\
\nonumber&=&\int_{s_l}^{s_{l+1}}dr\frac{1}{2\pi r}\left[\sum_{i=0}^{l-1}\int_{C_i}dx\rho^L_N(x)\left(1-\frac{N_i}{N\rho^L_N(C_l)}\right)\right.
\\
\nonumber&+&\left.\int_{s_l<|x|<r}dx\rho^L_N(x)\left(1-\frac{N_l}{N\rho^L_N(C_l)}\right)\right]^2
\\
\nonumber&=&\int_{C_l}dr\frac{1}{2\pi r}\left[\frac{N\rho^L_N(C_l)-P_l}{N}+\frac{N\rho^L_N(C_l)-N_l}{N}\frac{r^2-s_l^2}{s_{l+1}^2-s_l^2}\right]^2
\\
\label{okannuli3}&\leq& c\left[\frac{(N\rho^L_N(C_l)-P_l)^2}{N^2}+\frac{(N\rho^L_N(C_l)-N_l)^2}{N^2}\right]
\end{eqnarray}
where in the last inequality we have used that if $l=0$ the first summand disappears since $P_l=N\rho^L_N(P_l)=0$ almost everywhere and therefore the function $2\pi r$ in the denominator is multiplied for $r^4$, thus it is integrable.
Moreover, thanks to the choice of $\{\alpha_l\}_{l=0}^{L-1}$ we have
\begin{eqnarray*}
\rho^L_N(C_l)(1-\rho^L_N(C_l))=\underbrace{\frac{N-\sum_{j=l}^{L-1}N^{\alpha_j}}{N}}_{\leq 1}\underbrace{\frac{\sum_{j=l}^{L-1}N^{\alpha_j}}{N}}_{\leq c\frac{N^{\alpha_l}}{N}}\leq c\frac{N^{\alpha_l}}{N}=c\rho^L_N|_{C_l}|C_l|
\end{eqnarray*}
where $c$ depends on $L$. Therefore
\begin{eqnarray}
\label{okannuli1}\mathbb{E}\left[\frac{(N\rho^L_N(C_l)-P_l)^2}{N^2}\right]=\frac{\rho^L_N(C_l)(1-\rho^L_N(C_l))}{N}\leq c\frac{\rho^L_N|_{C_l}|C_l|}{N}
\end{eqnarray}
while
\begin{eqnarray*}
\label{okannuli9}\mathbb{E}\left[\frac{(N\rho^L_N(C_l)-N_l)^2}{N^2}\right]&=&\frac{\rho^L_N(C_l)(1-\rho^L_N(C_l))}{N}\leq c\frac{\rho^L_N|_{C_l}|C_l|}{N}
\end{eqnarray*}
Finally from \eqref{okannuli2}, \eqref{okannuli3} and \eqref{okannuli1} we get
\begin{eqnarray*}
\mathbb{E}\left[W_2^2\left(\rho^L_N,\sum_{l=0}^{L-1}\frac{N_l}{N}\frac{\rho^L_N\mathbbm{1}_{C_l}}{\rho^L_N(C_l)}\right)\right]\leq c\sum_{l=0}^{L-1}\frac{\rho^L_N|_{C_l}|C_l|}{\rho^L_N|_{C_l}N}=\frac{c|C|}{N}
\end{eqnarray*}

Then, the proof of \eqref{okannuli4} is exactly the same as \eqref{okannuli5}.

To obtain \eqref{okannuli6} it is sufficient to observe that in $A_{\theta}$
\begin{eqnarray*}
&&W_2^2\left(\sum_{l=0}^{L-1}\frac{3\theta M_l}{N}\frac{\rho^L_N\mathbbm{1}_{C_l}}{\rho^L_N(C_l)},\sum_{l=0}^{L-1}\frac{M_l-N_l+3\theta M_l}{N}\frac{\rho^L_N\mathbbm{1}_{C_l}}{\rho^L_N(C_l)}\right)
\\
&\leq&6\theta W_2^2\left(\sum_{l=0}^{L-1}\frac{M_l}{N}\frac{\rho^L_N\mathbbm{1}_{C_l}}{\rho^L_N(C_l)},\rho^L_N\right)
\\
&+&6\theta W_2^2\left(\rho^L_N,\sum_{l=0}^{L-1}\frac{M_l-N_l}{3\theta N}+\frac{M_l}{N}\frac{\rho^L_N\mathbbm{1}_{C_l}}{\rho^L_N(C_l)}\right)
\end{eqnarray*}
and then we use \eqref{okannuli5} and \eqref{okannuli4}.

\end{proof}

Now we can prove Theorem \ref{multiscalingthm}.

Thanks to Lemma \ref{semibip} it is sufficient to prove the upper bound for semidiscrete matching, in Proposition \ref{uppermulti}, and the lower bound for bipartite matching, in Proposition \ref{lowermulti}.

The structure of the proofs is the same as Theorem 1 in \cite{BC} and Theorems 1.1 and 1.2 in \cite{AGT}. First, we use the fact that the total transportation cost on the disk $C$ is estimated by the sum of the costs on the annuli $C_l$. This is possible thanks to Lemma \ref{okannuli}. Then we use the fact that the problem on the annulus $C_l$ has been solved in \cite{AGT} (it is a particular case of Theorem 1.1 and 1.2) because the probability density $\rho^L_N$ is piecewise constant on the annuli $C_l$ (and, thus, piecewise bounded from below). Therefore, if $N_l$ is the number of particles in $C_l$, each annulus contributes to the total cost with a term approximated by
$$
|C_l|\frac{\log N_l}{N_l}\approx|C_l|\frac{\log \mathbb{E}(N_l)}{N_l}
$$
except for a factor $4\pi$ or $2\pi$ in semidiscrete and bipartite matching respectively. The total cost is a convex combination of all these terms, so the main contributions (avoiding the factors $4\pi$ and $2\pi$) turns out to be
$$
\sum_{l=0}^{L-1}|C_l|\frac{\log\mathbb{E}(N_l)}{N}\approx\frac{\log N}{N}\sum_{l=0}^{L-1}\alpha_l|C_l|
$$
\begin{proposition}\label{uppermulti}
Let $X_1,\dots,X_N$ be independent random variables in $C$ with common distribution $\rho^L_N$. Then
$$
\limsup_{N\to\infty}\frac{N}{\log N}\mathbb{E}\left[W_2^2\left(\frac{1}{N}\sum_{i=1}^N\delta_{X_i},\rho^L_N\right)\right]\leq\frac{1}{4\pi}\sum_{l=0}^{L-1}\alpha_l|C_l|
$$
\end{proposition}

\begin{proof} Let $N_l$ be the number of points $X_i$ in $C_l$, i.e. $N_l:=\sum_{i=1}^N\mathbbm{1}(X_i\in C_l)$. Then, if we define
\begin{eqnarray*}
\frac{1}{N}\sum_{i=1}^N\delta_{X_i}=\sum_{l:N_l>0}\frac{N_l}{N}\frac{1}{N_l}\sum_{i: X_i\in C_l}\delta_{X_i}=:\sum_{l:N_l>0}\frac{N_l}{N}\mu^{N_l}
\end{eqnarray*}
Hence, thanks to triangle inequality and convexity of quadratic Wasserstein distance, if $\beta>0$
\begin{eqnarray}
\nonumber W_2^2\left(\frac{1}{N}\sum_{i=1}^N\delta_{X_i},\rho^L_N\right)&=&W_2^2\left(\sum_{l:N_l>0}\frac{N_l}{N}\mu^{N_l},\rho^L_N\right)
\\
\nonumber&\leq&(1+\beta)\sum_{l:N_l>0}\frac{N_l}{N}W_2^2\left(\mu^{N_l},\frac{\rho^L_N\mathbbm{1}_{C_l}}{\rho^L_N(C_l)}\right)
\\
\label{uppermulti1}&+&\frac{1+\beta}{\beta}W_2^2\left(\sum_{l=0}^{L-1}\frac{N_l}{N}\frac{\rho^L_N\mathbbm{1}_{C_l}}{\rho^L_N(C_l)},\rho^L_N\right)
\end{eqnarray}
and, since $\beta>0$ is arbitrary, combining \eqref{uppermulti1} with \eqref{okannuli5} of Lemma \ref{okannuli} we get
\begin{eqnarray}
&&\nonumber\limsup_{N\to\infty}\frac{N}{\log N}\mathbb{E}\left[W_2^2\left(\frac{1}{N}\sum_{i=1}^N\delta_{X_i},\rho^L_N\right)\right]
\\
\label{uppermulti10}&\leq&\limsup_{N\to\infty}\frac{N}{\log N}\mathbb{E}\left[\sum_{l:N_l>0}\frac{N_l}{N}W_2^2\left(\mu^{N_l},\frac{\rho^L_N\mathbbm{1}_{C_l}}{\rho^L_N(C_l)}\right)\right]
\end{eqnarray}
Let now $A_l$ be defined as
$$
A_l:=\left\{N_l\geq \frac{N\rho^L_N(C_l)}{2}\right\}=\left\{N_l\geq\frac{\mathbb{E}(N_l)}{2}\right\}
$$
We can compute the expected value in \eqref{uppermulti10} separately in the sets $A_l^c$ and $A_l$.

In $A_l^c$ we have the bound
\begin{eqnarray}
\label{uppermulti3}\frac{N_l}{N}W_2^2\left(\mu^{N_l},\frac{\rho^L_N\mathbbm{1}_{C_l}}{\rho^L_N(C_l)}\right)\leq c\rho^L_N(C_l)
\end{eqnarray}
while
\begin{eqnarray}
\label{uppermulti4}\mathbb{P}(A_l^c)\leq\mathbb{P}\left(|N_l-N\rho^L_N(C_l)|\geq\frac{\rho^L_N(C_l)}{2}\right)\leq\frac{4}{N\rho^L_N(C_l)}
\end{eqnarray}
Using \eqref{uppermulti3} and \eqref{uppermulti4} we get
\begin{eqnarray}
\label{uppermulti5}\mathbb{E}\left[\sum_{l:N_l>0}\frac{N_l}{N}W_2^2\left(\mu^{N_l},\frac{\rho^L_N\mathbbm{1}_{C_l}}{\rho^L_N(C_l)}\right)\mathbbm{1}_{A_l^c}\right]\leq\frac{c}{N}
\end{eqnarray}
Therefore we can limit ourselves to consider the expected value in $A_l$: we use the properties of the conditioned expected value and \eqref{teoremone1} of Theorem \ref{teoremone}.

Indeed, since $\min_{l=0,\dots,L-1}N\rho^L_N(C_l)\xrightarrow[N\to\infty]{}\infty$ there exists a function $\omega(N)\xrightarrow[N\to\infty]{}0$ such that
\begin{eqnarray}
\nonumber&\mathbb{E}&\left[\sum_{l:N_l>0}\frac{N_l}{N}W_2^2\left(\mu^{N_l},\frac{\rho^L_N\mathbbm{1}_{C_l}}{\rho^L_N(C_l)}\right)\mathbbm{1}_{A_l}\right]
\\
\nonumber&=&\mathbb{E}\left[\sum_{l=0}^{L-1}\frac{N_l}{N}\mathbbm{1}_{A_l}\mathbb{E}_{N_l}\left[W_2^2\left(\mu^{N_l},\frac{\rho^L_N\mathbbm{1}_{C_l}}{\rho^L_N(C_l)}\right)\right]\right]
\\
\label{uppermulti6}&\leq&\mathbb{E}\left[\sum_{l=0}^{L-1}\frac{\log N_l}{4\pi N}|C_l|(1+\omega(N))\right]
\end{eqnarray}
Finally, we use the concavity of the function $\log x$ to observe that
\begin{eqnarray}
\label{uppermulti8}\mathbb{E}\left[\sum_{l=0}^{L-1}\frac{\log N_l}{4\pi N}|C_l|\right]\leq\sum_{l=0}^{L-1}\frac{\log\mathbb{E}(N_l)}{4\pi N}|C_l|\leq\frac{\log N}{4\pi N}\sum_{l=0}^{L-1}\alpha_l|C_l|
\end{eqnarray}
Thus, using \eqref{uppermulti5}, \eqref{uppermulti6} and \eqref{uppermulti8} we obtain
\begin{eqnarray*}
\mathbb{E}\left[\sum_{l:N_l>0}\frac{N_l}{N}W_2^2\left(\mu^{N_l},\frac{\rho^L_N\mathbbm{1}_{C_l}}{\rho^L_N(C_l)}\right)\right]\leq\frac{\log N}{4\pi N}\sum_{l=0}^{L_1}\alpha_l|C_l|(1+\omega(N))+\frac{c}{N}
\end{eqnarray*}
and combining this with \eqref{uppermulti10} we obtain the thesis.

\end{proof}

\begin{proposition}\label{lowermulti}
Let $X_1,\dots,X_N$ and $Y_1,\dots,Y_N$ independent random variables in $C$ with common distribution $\rho^L_N$. Then it holds
$$
\liminf_{N\to\infty}\frac{N}{\log N}\mathbb{E}\left[W_2^2\left(\frac{1}{N}\sum_{i=1}^N\delta_{X_i},\frac{1}{N}\sum_{i=1}^N\delta_{Y_i}\right)\right]\geq\frac{1}{2\pi}\sum_{l=0}^{L-1}\alpha_l|C_l|
$$
\end{proposition}

\begin{proof} \textit{(Sketch)} Since the proof is quite similar to the previous one, we only explain the differences.

First, we can restrict to a special set
$$
A_{\theta}:=\left\{\begin{array}{lcr}|N_l-N\rho^L_N(C_l)|&\leq&\theta N\rho^L_N(C_l)\\|M_l-N\rho^L_N(C_l)|&\leq&\theta N\rho^L_N(C_l)\end{array}\qquad\forall l=0,\dots,L\right\}
$$
where $\theta=\frac{1}{\sqrt{\log N}}$. Its complementary has small probability. Then, if $N_l$ and $M_l$ are respectively the number of particles $X_i$ and $Y_i$ in $C_l$, we rename
$$
\mu^{N_l}:=\frac{1}{N_l}\sum_{i:X_i\in C_l}\delta_{X_i}\qquad;\qquad\nu^{M_l}:=\frac{1}{M_l}\sum_{i:Y_i\in C_l}\delta_{Y_i}
$$
Using the superadditivity of $Wb_2^2$, we obtain
\begin{eqnarray}
\nonumber&\mathbb{E}&\left[W_2^2\left(\frac{1}{N}\sum_{i=1}^N\delta_{X_i},\frac{1}{N}\sum_{i=1}^N\delta_{Y_i}\right)\right]
\\
\label{lowermulti5}&\geq&(1-\beta)\mathbb{E}\left[\sum_{l=0}^{L-1}\frac{N_l}{N}\mathbb{E}\left[Wb_2^2(\mu^{N_l},\nu^{M_l})\right]\mathbbm{1}_{A_{\theta}}\right]
\\
\label{lowermulti6}&-&\frac{1-\beta}{\beta}\mathbb{E}\left[W_2^2\left(\sum_{l=0}^{L-1}\frac{N_l}{N}\nu^{M_l},\sum_{l=0}^{L-1}\frac{M_l}{N}\nu^{M_l}\right)\mathbbm{1}_{A_{\theta}}\right]
\end{eqnarray}
The main contribution is given by the term in \eqref{lowermulti5}, and as in the previous Theorem it can be estimated using \eqref{teoremone2} of Theorem \ref{teoremone}.

To prove that the term in \eqref{lowermulti6} is negligible, it is sufficient to recall Lemmas \ref{forbipartitematching} and \ref{okannuli}.

Letting $\beta\to0$, we get the thesis.
\end{proof}

Now we have proved Theorem \ref{multiscalingthm}. Let us observe that we can choose the exponents $\alpha_l$ and the annuli $C_l$ in an interesting way. If 
$$
0=c_0<c_1<\dots<c_{L-1}<c_L=1\quad;\quad s_l=c_lS
$$
and 
$$
\alpha_l=1-c_l^2\quad;\quad l=0,\dots,L
$$
we obtain
$$
\sum_{l=0}^{L-1}\alpha_l|C_l|=\pi S^2\sum_{l=0}^{L-1}\alpha_l(c_{l+1}^2-c_l^2)=\pi S^2\sum_{l=0}^{L-1}\alpha_l(\alpha_l-\alpha_{l+1})
$$
that is a Riemann sum for the function $f(x)=\pi S^2x$. Therefore, if we decrease $\max_{l=0,\dots,L-1}\{\alpha_l-\alpha_{l+1}\}$, in the limit we obtain
$$
\pi S^2\int_0^1dxx=\frac{\pi S^2}{2}
$$
In particular, the case $S=\sqrt{2\log N}$ introduces us to Section \ref{sec3}, indeed the Gaussian density has the following property: if $\alpha_0=1>\alpha_1>\dots>\alpha_L=0$, the annulus $A_l$ defined by
$$
A_l:=\{\sqrt{2(1-\alpha_l)\log N}\leq|x|\leq\sqrt{2(1-\alpha_{l+1})\log N}\}
$$
has measure $2\pi(\alpha_l-\alpha_{l+1})\log N$. Moreover the averaged number of particles in $A_l$ is $N^{\alpha_l}-N^{\alpha_{l+1}}\sim N^{\alpha_l}$, therefore it contributes to the total cost with
$$
\frac{1}{2\pi}2\pi(\alpha_l-\alpha_{l+1})\log N\log[N^{\alpha_l}]=\alpha_l(\alpha_l-\alpha_{l+1})(\log N)^2
$$
and by summing over $l=0,\dots,L-1$ and decreasing $\max_{l=0,\dots,L-1}\{\alpha_l-\alpha_{l+1}\}$ we get
$$
\frac{1}{2}(\log N)^2
$$
Notwithstanding this, the case of the Gaussian density have some further difficulties.

The first one is that what we have proved in the case of the multiscaling density depends on the number of the annuli we are considering, while we are currently not able to approximate the gaussian density with a piecewise multiscaling density on a finite number of annuli. Otherwise, if we consider a countable set of annuli,  we need a uniform bound for the cost of the problem on an annulus. Instead, when dividing the disk of radius $\sqrt{2\log N}$ into squares, using a rescaling argument, we only need a bound for the cost of the problem on a square (and we already have it from \cite{AST}). Moreover, the main property we use here is not that the Gaussian density is a radial density, but rather that it is a product measure, i.e. a function of $x_1$ times a function of $x_2.$
\section{The Gaussian density}\label{sec3}

This Section concerns the problem of $X_1,\dots,X_N$ and $Y_1,\dots,Y_N$ independent random variables in $\mathbb{R}^2$ distributed according to Gaussian measure $\rho$, that is
$$
\rho(x):=\mu(x_1)\mu(x_2)\quad;\quad\mu(z):=\frac{e^{-\frac{z^2}{2}}}{\sqrt{2\pi}}
$$
In Subsection \ref{subsec7} we prove the following theorem.
\begin{theorem}\label{gauxthm}
If $X_1,\dots,X_N$ and $Y_1,\dots,Y_N$ are iid random variables distributed with the Gaussian density $\rho$, it holds 
\begin{eqnarray*}
&&\frac{N}{(\log N)^2}\mathbb{E}\left[W_2^2\left(\frac{1}{N}\sum_{i=1}^N\delta_{X_i},\rho\right)\right]\xrightarrow[N\to\infty]{}\frac{1}{4}
\\
&&\frac{N}{(\log N)^2}\mathbb{E}\left[W_2^2\left(\frac{1}{N}\sum_{i=1}^N\delta_{X_i},\frac{1}{N}\sum_{i=1}^N\delta_{Y_i}\right)\right]\xrightarrow[N\to\infty]{}\frac{1}{2}
\end{eqnarray*}
\end{theorem}
First, we underline that also if Gaussian density has an unbounded support, the number of particles in a cube of side $dx$ is $Ne^{-{\frac{|x|^2}{2}}}dx$, and we can notice that $Ne^{-\frac{|x|^2}{2}}$ is strictly smaller than 1 when $|x|>\sqrt{2\log N}$. Therefore, using the results in \cite{BC} and \cite{AGT}, we can suppose that the cost for semidiscrete and bipartite matching is
$$
\frac{1}{N}\int_{|x|\leq\sqrt{2\log N}}dx\log(N\rho(x))=\pi\frac{(\log N)^2}{N}+\frac{{\cal O}(1)}{N}
$$
except for a factor $\frac{1}{4\pi}$ or $\frac{1}{2\pi}$ respectively.

To achieve these result, we apply a cut-off and we substitute $\rho$ with a density that we will call again $\rho_N$ and whose support is contained in $\{|x|\leq\sqrt{2\log N}\}$. To define $\rho_N$, we proceed in the following way.  We cannot arrive exactly at $\sqrt{2\log N}$, otherwise there would be too few particles close to the boundary of $\{|x|\leq\sqrt{2\log N}\}$, therefore we define
$$
r_N:=\sqrt{2\log\left(\frac{N}{(\log N)^{\alpha}}\right)}\quad;\quad1<\alpha<2
$$
and we construct a collection of squares that covers $\{|x|\leq r_N\}$, in this way
\begin{eqnarray*}
{\cal J}&:=&\left\{\left.(a_j,a_{j+1}):=\left(\frac{j\epsilon}{r_N},\frac{(j+1)\epsilon}{r_N}\right)\right|j\in\mathbb{Z}\right\}
\\
{\cal K}&:=&\left\{\left.(b_k,b_{k+1}):=\left(\frac{k\epsilon}{r_N},\frac{(k+1)\epsilon}{r_N}\right)\right|k\in\mathbb{Z}\right\}
\end{eqnarray*}
${\cal J}$ is a set of intervals in direction $x_1$, while ${\cal K}$ is a set of intervals in direction $x_2$. Now we define a set of squares that covers $\{|x|\leq r_N\}$, as follows. First, we denote by $k_{\min}$ and $k_{\max}$ and by $j_k^{\min}$ and $j_k^{\max}$
\begin{eqnarray*}
k_{\min}&:=&-\left\lfloor\frac{r_N^2}{\epsilon}\right\rfloor-1\quad;\quad k_{\max}:=\left\lfloor\frac{r_N^2}{\epsilon}\right\rfloor
\\
j_k^{\min}&:=&\inf\{j\in\mathbb{Z}|(a_j,a_{j+1})\times(b_k,b_{k+1})\cap\{|x|\leq r_N\}\neq\emptyset\}
\\
j_k^{\max}&:=&\sup\{j\in\mathbb{Z}|(a_j,a_{j+1})\times(b_k,b_{k+1})\cap\{|x|\leq r_N\}\neq\emptyset\}
\end{eqnarray*}
and then we define ${\cal Q}$ as the minimal set of squares that covers $\{|x|\leq r_N\}$
\begin{eqnarray*}
{\cal Q}:=\left\{Q^j_k:=(a_j,a_{j+1})\times(b_k,b_{k+1})\left|\begin{array}{lc}k_{\min}\leq k\leq k_{\max}\\j_k^{\min}\leq j\leq j_k^{\max}\end{array}\right\}\right.
\end{eqnarray*}
\begin{figure}[h!]
\centering
\includegraphics[width=7cm]{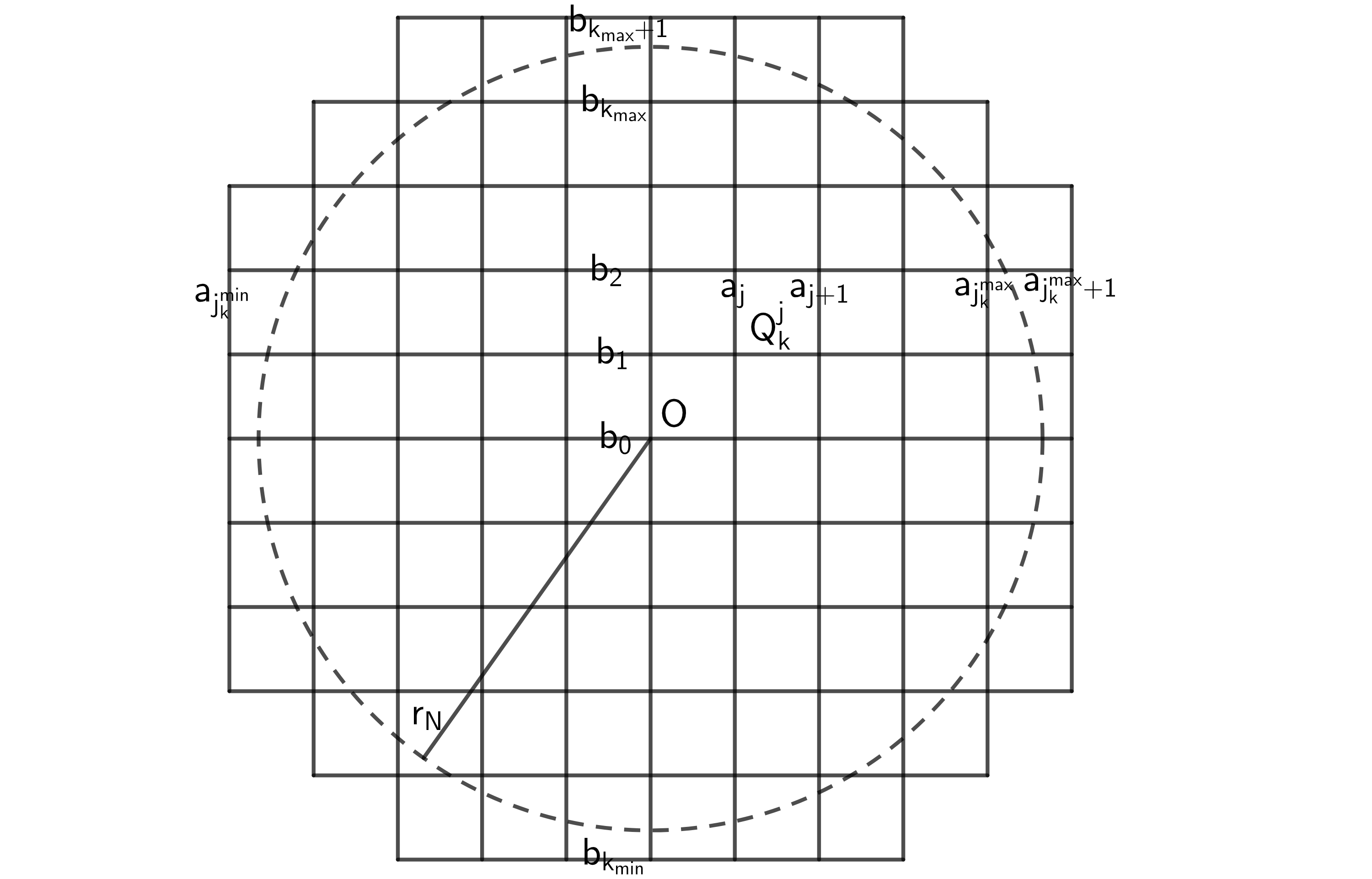}
\caption{The set of squares $Q^j_k$ where the cut-off is applied.}
\end{figure}
Before going on, here we can notice that, thanks to the choice of the squares, if $N^j_k$ is the (random) number of points in the square $Q^j_k$ when the distribution of the particles is Gaussian (after having applied the cut-off, the expectation of this number can only increase) we have
\begin{eqnarray*}
\mathbb{E}(N^j_k)=N\int_{Q^j_k}dx\rho(x)\geq N\left(\frac{\epsilon}{r_N}\right)^2\frac{e^{-\frac{\left(r_N+\frac{\epsilon\sqrt{2}}{r_N}\right)^2}{2}}}{\sqrt{2\pi}}\geq c\epsilon^2(\log N)^{\alpha-1}\xrightarrow[N\to\infty]{}\infty
\end{eqnarray*}
Let also ${\cal R}$ be set of horizontal rectangles that covers $\{|x|\leq r_N\}$
$$
{\cal R}:=\left\{R_k:=\bigcup_{j_k^{\min}\leq j\leq j_k^{\max}} Q^j_k\right\}
$$
with projections $J^k$ on the axis $x_1$, where each $J^k$ is defined by
$$
J^k:=\bigcup_{j_k^{\min}\leq j\leq j_k^{\max}} (a_j,a_{j+1})
$$
\begin{figure}[h!]
\centering
\includegraphics[width=7cm]{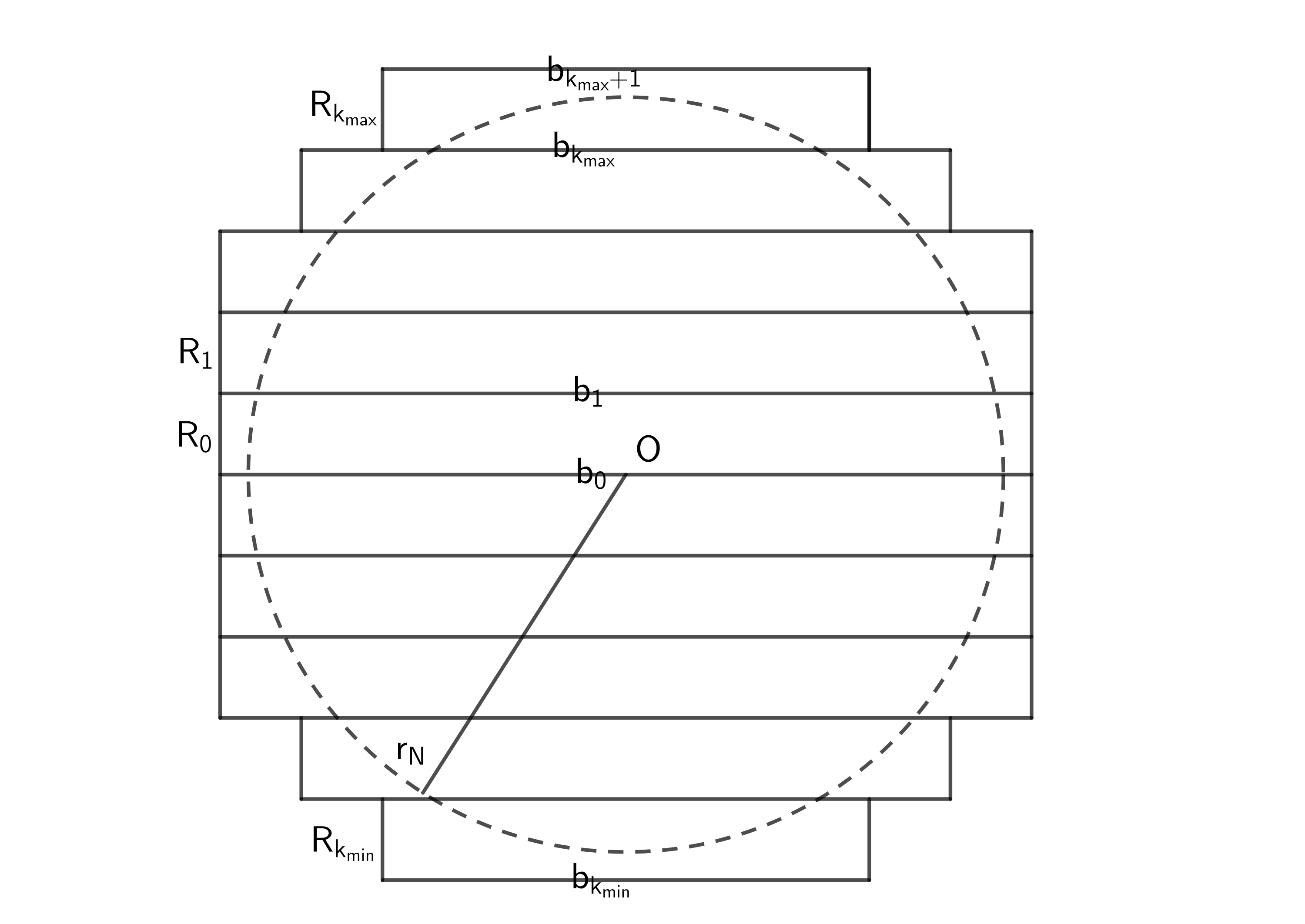}
\caption{The set of rectangles where the cut-off is applied, except for zero measure sets.}
\end{figure}
Finally, we define $E_N$ 
$$
E_N:=\bigcup_{Q^j_k\in{\cal Q}}Q^j_k
$$
and $\rho_N$ is the Gaussian measure restricted to $E_N$
$$
\rho_N:=\frac{\rho\mathbbm{1}_{E_N}}{\rho(E_N)}
$$
Hereafter, if $\tilde X_1,\dots,\tilde X_N$ and $\tilde Y_1,\dots,\tilde Y_N$ are independent and identically distributed with measure $\rho_N,$ and we define
$$
N^j_k:=\sum_{i=1}^N\mathbbm{1}(\tilde X_i\in Q^j_k)\quad;\quad M^j_k:=\sum_{i=1}^N\mathbbm{1}(\tilde Y_i \in Q^j_k)
$$
as the number of points $\tilde X_i$ and $\tilde Y_i$ in the square $Q^j_k,$ respectively

Finally, where not specified, we denote $\sum_{j,k}:=\sum_{k=k_{\min}}^{k_{\max}}\sum_{j=j_k^{\min}}^{j_k^{\max}}$.

In Subsection \ref{subsec6} we prove some bounds that we will need for the proof of Theorem \ref{gauxthm} in Subsection \ref{subsec7}.

\subsection{Preliminary estimates}\label{subsec6}
This first result proves that we can substitute $N$ independent random variables with common distribution $\rho$ with $N$ independent random variables with common distribution $\rho_N$.
\begin{lemma}\label{cutoff}
Let $\rho$ and $\rho_N$ defined as before, $X_1,\dots,X_N$ independent random variables in $\mathbb{R}^2$ with common distribution $\rho$ and $T:\mathbb{R}^2\to\mathbb{R}^2$ the optimal map that transports $\rho$ in $\rho_N$. Then we have
\begin{eqnarray}
&&\label{cutoff1}W_2^2\left(\rho,\rho_N\right)\leq c\frac{(\log N)^{\alpha}}{N}
\\
&&\label{cutoff2}\mathbb{E}\left[W_2^2\left(\frac{1}{N}\sum_{i=1}^N\delta_{X_i},\frac{1}{N}\sum_{i=1}^N\delta_{T(X_i)}\right)\right]\leq c\frac{(\log N)^{\alpha}}{N}
\end{eqnarray}
\end{lemma}

\begin{proof} For \eqref{cutoff1} we use again Theorem \ref{talfor} to write
\begin{eqnarray*}
W_2^2\left(\rho,\rho_N\right)&\leq&2\int_{E_N}dx\rho(x)\frac{\rho_N(x)}{\rho(x)}\log\left(\frac{\rho_N(x)}{\rho(x)}\right)
\\
&=&2\log\left(\frac{1}{\rho(E_N)}\right)\leq2\log\left(\frac{1}{\rho(\{|x|\leq r_N\})}\right)\leq c\frac{(\log N)^{\alpha}}{N}
\end{eqnarray*}
while as for \eqref{cutoff2}, if $T$ is the optimal map that transports $\rho$ in $\rho_N$, that is
\begin{eqnarray*}
\rho_N(E)&=&\rho(T^{-1}(E))
\\
W_2^2\left(\rho,\rho_N\right)&=&\int_{\mathbb{R}^2}dx\rho(x)|x-T(x)|^2
\end{eqnarray*}
we have
\begin{eqnarray*}
&&\mathbb{E}\left[W_2^2\left(\frac{1}{N}\sum_{i=1}^N\delta_{X_i},\frac{1}{N}\sum_{i=1}^N\delta_{T(X_i)}\right)\right]\leq\mathbb{E}\left[\frac{1}{N}\sum_{i=1}^N|X_i-T(X_i)|^2\right]
\\
&=&\int_{\mathbb{R}^2}dx\rho(x)|x-T(x)|^2=W_2^2(\rho,\rho_N)
\end{eqnarray*}
that is the thesis thanks to \eqref{cutoff1}.

\end{proof}

The following Proposition allows us to compute the total cost of the problem as the sum of the costs of the problems on the squares $Q^j_k$. We have to bound the expectation of the distance between the Gaussian measure and the same Gaussian measure modified on the squares $Q^j_k$ by a factor $\frac{N^j_k}{N\rho(Q^j_k)}$. Therefore the two measures we are considering are $\rho$ and $\sum_{j,k}\frac{N^j_k}{N}\frac{\rho\mathbbm{1}_{Q^j_k}}{\rho(Q^j_k)}$.

The reason why these measures should be similar is that $N^j_k$ is very close to its expectation, that is $\mathbb{E}(N^j_k)=N\rho(Q^j_k)$. To prove it, we proceed in two steps and use the triangle inequality between the two measure involved and a third measure, that is $
\sum_k\frac{N_k}{N}\frac{\rho\mathbbm{1}_{R_k}}{\rho(R_k)}$, where $N_k$ is the number of points $\tilde X_i$ in the rectangle $R_k$.

As for the distance between $\sum_{j,k}\frac{N^j_k}{N}\frac{\rho\mathbbm{1}_{Q^j_k}}{\rho(Q^j_k)}$ and $\sum_k\frac{N_k}{N}\frac{\rho\mathbbm{1}_{R_k}}{\rho(R_k)}$, first we use convexity of Wasserstein distance to restrict the problem to the rectangles, indeed we have
$$
W_2^2\left(\sum_{j,k}\frac{N^j_k}{N}\frac{\rho\mathbbm{1}_{Q^j_k}}{\rho(Q^j_k)},\sum_k\frac{N_k}{N}\frac{\rho\mathbbm{1}_{R_k}}{\rho(R_k)}\right)\leq\sum_{k:N_k>0}\frac{N_k}{N}W_2^2\left(\frac{N^j_k}{N_k}\frac{\rho\mathbbm{1}_{Q^j_k}}{\rho(Q^j_k)},\frac{\rho\mathbbm{1}_{R_k}}{\rho(R_k)}\right)
$$
Then, we argue as in Lemmas \ref{okannuli}: when using Benamou-Brenier formula there is a vanishing density in the denominator, and this causes a divergent term. But this divergent term is completely balanced from the fluctuations of the particles, which are very few.

Then, we use again Talagrand formula to bound the distance between $\sum_k\frac{N_k}{N}\frac{\rho\mathbbm{1}_{R_k}}{\rho(R_k)}$ and $\rho$.
\begin{proposition}\label{oksquares}
There exist a constant $c>0$ such that if $\tilde X_1,\dots,\tilde X_N$ and $\tilde Y_1,\dots,\tilde Y_N$ are independent random variables with common distribution $\rho_N$ and if
\begin{eqnarray*}
A_{\theta}:=\bigcap_{j,k}\left\{\begin{array}{lcr}|N^j_k-N\rho_N(Q^j_k)|&\leq&\theta N\rho_N(Q^j_k)\\|M^j_k-N\rho_N(Q^j_k)|&\leq&\theta N\rho_N(Q^j_k)\end{array}\right\}
\end{eqnarray*}
then
\begin{eqnarray}
\label{oksquares3}&&\mathbb{E}\left[W_2^2\left(\sum_{j,k}\frac{N^j_k}{N}\frac{\rho\mathbbm{1}_{Q^j_k}}{\rho(Q^j_k)},\rho\right)\right]\leq c\frac{(\log N)^{\frac{3}{2}}}{\epsilon N}+c\frac{(\log N)^{\alpha}}{N}
\\
\nonumber&&\mathbb{E}\left[W_2^2\left(\sum_{j,k}\left(\frac{M^j_k-N^j_k}{3\theta N}+\frac{M^j_k}{N}\right)\frac{\rho\mathbbm{1}_{Q^j_k}}{\rho(Q^j_k)},\rho\right)\mathbbm{1}_{A_{\theta}}\right]
\\
\label{oksquares4}&&\leq c\frac{(\log N)^{\frac{3}{2}}}{\epsilon\theta^2N}+c\frac{(\log N)^{\alpha}}{N}
\\
\nonumber&&\mathbb{E}\left[W_2^2\left(\sum_{j,k}\frac{3\theta M^j_k}{N}\frac{\rho\mathbbm{1}_{Q^j_k}}{\rho(Q^j_k)},\sum_{j,k}\frac{M^j_k-N^j_k+3\theta M^j_k}{N}\frac{\rho\mathbbm{1}_{Q^j_k}}{\rho(Q^j_k)}\right)\mathbbm{1}_{A_{\theta}}\right]
\\
\label{oksquares5}&&\leq c\frac{(\log N)^{\frac{3}{2}}}{\epsilon\theta N}+c\theta\frac{(\log N)^{\alpha}}{N}
\end{eqnarray}
\end{proposition}

\begin{proof}
We only prove \eqref{oksquares3} and \eqref{oksquares5}, indeed \eqref{oksquares4} is analogue to \eqref{oksquares3}.

We start by proving \eqref{oksquares3}, therefore we define
$$
N_k:=\sum_{j=j_k^{\min}}^{j_k^{\max}}N^j_k
$$
so that $N_k$ is the numbers of particles $X_i$ in the whole rectangle $R_k$ and
$$
P^j_k:=\sum_{i=j_k^{\min}}^{j-1}N^i_k
$$
Therefore $P^j_k$ is the numbers of particles $X_i$ in $R_k$ not in the whole rectangle, but rather only until $a_j$.

As for \eqref{oksquares4}, the only difference in the proof is that where we have summands involving $N^j_k$ we will find the same terms involving $M^j_k+\frac{N^j_k-M^j_k}{\theta}$ (if $M^j_k$ is the number of points $Y_i$ in the square $Q^j_k$).

We proceed in two steps. First, we focus on the distance between the density modified on all the squares $Q^j_k$ and the one modified only on the rectangles $R_k$; then, we study the distance between the measure modified on the rectangles $R_k$ and the Gaussian measure itself.

Using first the triangle inequality and then the convexity of quadratic Wasserstein distance we get
\begin{eqnarray}
\label{oksquares1}W_2^2\left(\sum_{j,k}\frac{N^j_k}{N}\frac{\rho\mathbbm{1}_{Q^j_k}}{\rho(Q^j_k)},\rho\right)&\leq&2\sum_{k:N_k>0}\frac{N_k}{N}W_2^2\left(\sum_{j=j_k^{\min}}^{j_k^{\max}}\frac{N^j_k}{N_k}\frac{\rho\mathbbm{1}_{Q^j_k}}{\rho(Q^j_k)},\frac{\rho\mathbbm{1}_{R_k}}{\rho(R_k)}\right)
\\
\label{oksquares2}&+&2W_2^2\left(\sum_k\frac{N_k}{N}\frac{\rho\mathbbm{1}_{R_k}}{\rho(R_k)},\rho\right)
\end{eqnarray}
As for the term in \eqref{oksquares1}, we observe that we are considering again product measures in the rectangle $R_k$ whose marginals coincide in the direction $x_2$, indeed we have
\begin{eqnarray*}
\sum_{j=j_k^{\min}}^{j_k^{\max}}\frac{N^j_k}{N_k}\frac{\rho(x)\mathbbm{1}_{Q^j_k}(x)}{\rho(Q^j_k)}&=&\sum_{j=j_k^{\min}}^{j_k^{\max}}\frac{N^j_k}{N_k}\frac{\mu(x_1)\mathbbm{1}_{(a_j,a_{j+1})}(x_1)}{\mu(a_j,a_{j+1})}\frac{\mu(x_2)\mathbbm{1}_{(b_k,b_{k+1})}(x_2)}{\mu(b_k,b_{k+1})}
\\
\frac{\rho(x)\mathbbm{1}_{R_k}(x)}{\rho(R_k)}&=&\frac{\mu(x_1)\mathbbm{1}_{J^k}(x_1)}{\mu(J^k)}\frac{\mu(x_2)\mathbbm{1}_{(b_k,b_{k+1})}(x_2)}{\mu(b_k,b_{k+1})}
\end{eqnarray*}
therefore we just have a one dimensional problem: thanks to Lemma \ref{productproperty} we get
\begin{eqnarray}
\label{oksquares6}W_2^2\left(\sum_{j=j_k^{\min}}^{j_k^{\max}}\frac{N^j_k}{N_k}\frac{\rho\mathbbm{1}_{Q^j_k}}{\rho(Q^j_k)},\frac{\rho\mathbbm{1}_{R_k}}{\rho(R_k)}\right)\leq W_2^2\left(\sum_{j=j_k^{\min}}^{j_k^{\max}}\frac{N^j_k}{N_k}\frac{\mu\mathbbm{1}_{(a_j,a_{j+1})}}{\mu(a_j,a_{j+1})},\frac{\mu\mathbbm{1}_{J^k}}{\mu(J^k)}\right)
\end{eqnarray}
To bound this term we argue as in Lemma \ref{okannuli}: we define $f:J^k\to\mathbb{R}$ such that if $x_1\in (a_j,a_{j+1})$
$$
f'(x_1)=\frac{P^j_k-\mathbb{E}_{N_k}(P^j_k)}{N_k}+\frac{N^j_k-\mathbb{E}_{N_k}(N^j_k)}{N_k}\frac{\mu(a_j,x_1)}{\mu(a_j,a_{j+1})}
$$
so that, thanks to $\mathbb{E}_{N_k}(N^j_k)=N_k\frac{\mu(a_j,a_{j+1})}{\mu(J^k)}$, $f$ is the weak solution of
$$
f''(x_1)=\sum_{j=j_k^{\min}}^{j_k^{\max}}\left(\frac{N^j_k}{N_k}\frac{\mu(x_1)}{\mu(a_j,a_{j+1})}-\frac{\mu(x_1)}{\mu(J^k)}\right)\mathbbm{1}_{(a_j,a_{j+1})}(x_1)
$$
that is the difference between the densities we are considering.

Then, thanks to Theorem \ref{benbrefor} we have
\begin{eqnarray}
\nonumber&&W_2^2\left(\sum_{j=j_k^{\min}}^{j_k^{\max}}\frac{N^j_k}{N_k}\frac{\mu\mathbbm{1}_{(a_j,a_{j+1})}}{\mu(a_j,a_{j+1})},\frac{\mu\mathbbm{1}_{J^k}}{\mu(J^k)}\right)\leq4\int_{J^k}dx_1\frac{(f'(x_1))^2}{\mu(x_1)}\mu(J^k)
\\
\nonumber&\leq& c\sum_{j=j_k^{\min}}^{j_k^{\max}}\left[\left(\frac{P^j_k-\mathbb{E}_{N_k}(P^j_k)}{N_k}\right)^2+\left(\frac{N^j_k-\mathbb{E}_{N_k}(N^j_k)}{N_k}\right)^2\right]\mu(J^k)\int_{a_j}^{a_{j+1}}dx_1\frac{e^{\frac{x_1^2}{2}}}{\sqrt{2\pi}}
\\
\label{oksquares7}&\leq& c\sum_{j=j_k^{\min}}^{j_k^{\max}}\left[\left(\frac{P^j_k-\mathbb{E}_{N_k}(P^j_k)}{N_k}\right)^2\right]\mu(J^k)(a_{j+1}-a_j)e^{\frac{a_j^2}{2}}
\\
\label{oksquares14}&+&c\sum_{j=j_k^{\min}}^{j_k^{\max}}\left[\left(\frac{N^j_k-\mathbb{E}_{N_k}(N^j_k)}{N_k}\right)^2\right]\mu(J^k)(a_{j+1}-a_j)e^{\frac{a_j^2}{2}}
\end{eqnarray}
where in the last inequality we have used that, thanks to the choice of the points $a_j$, we have $|a_{j+1}^2-a_j^2|\leq c$.

We are going to estimate these two terms using the fact that where the density is small there are also small fluctuations of the particles, so that there will be a balance between the fluctuations and the divergent terms. In particular, our aim is to show that all the terms in the last two sums perfectly balance, and only $(a_{j+1}-a_j)$ remains.

Thus, to bound \eqref{oksquares7} and \eqref{oksquares14} we observe that we can condition to the number of particles $X_i$ in $R_k$ (that is $N_k$) to obtain
\begin{eqnarray}
\label{oksquares9}\mathbbm{E}_{N_k}\left[\left(\frac{P^j_k-\mathbb{E}_{N_k}(P^j_k)}{N_k}\right)^2\right]&=&\frac{1}{N_k}\frac{\mu(a_{j_k^{\min}},a_j)}{\mu(J^k)}\frac{\mu(a_j,a_{j_k^{\max}+1})}{\mu(J^k)}
\\
\nonumber\mathbbm{E}_{N_k}\left[\left(\frac{N^j_k-\mathbb{E}_{N_k}(N^j_k)}{N_k}\right)^2\right]&=&\frac{1}{N_k}\frac{\mu(a_j,a_{j+1})}{\mu(J^k)}\left(1-\frac{\mu(a_j,a_{j+1})}{\mu(J^k)}\right)
\\
\label{oksquares12}&\leq&\frac{1}{N_k}\frac{\mu(a_j,a_{j+1})}{\mu(J^k)}
\end{eqnarray}
Before going on, here we observe that, when proving \eqref{oksquares4}, to bound \eqref{oksquares7} and \eqref{oksquares14} at this point we should condition both to $N_k$ and to $M_k$ (and not only to $N_k$). In this way instead of $\frac{1}{N_k}$ in \eqref{oksquares9} and \eqref{oksquares12}, up for a constant we would obtain
$$
\frac{1}{M_k+\frac{M_k-N_k}{3\theta}}+\frac{M_k+N_k}{\theta^2\left(M_k+\frac{M_k-N_k}{3\theta}\right)^2}
$$
and these term can be estimated (in $A_{\theta}$, and but for multiplicative constants) by
$$
\frac{1}{\theta^2\left(M_k+\frac{M_k-N_k}{3\theta}\right)}
$$
Finally, combining \eqref{oksquares6} with \eqref{oksquares7} and \eqref{oksquares14} we get
\begin{eqnarray}
\nonumber&&\mathbb{E}_{N_k}\left[W_2^2\left(\sum_{j=j_k^{\min}}^{j_k^{\max}}\frac{N^j_k}{N_k}\frac{\mu\mathbbm{1}_{(a_j,a_{j+1})}}{\mu(a_j,a_{j+1})},\frac{\mu\mathbbm{1}_{J^k}}{\mu(J^k)}\right)\right]
\\
\nonumber&\leq&\frac{c}{N_k}\sum_{j=j_k^{\min}}^{j_k^{\max}}\left[\mu(a_j,a_{j+1})+\frac{\mu(a_{j_k^{\min}},a_j)\mu(a_j,a_{j_k^{\max}+1})}{\mu(J^k)}\right]e^{\frac{a_j^2}{2}}(a_{j+1}-a_j)
\\
\label{oksquares8}
\end{eqnarray}
\begin{figure}[h!]
\centering
\includegraphics[width=7cm]{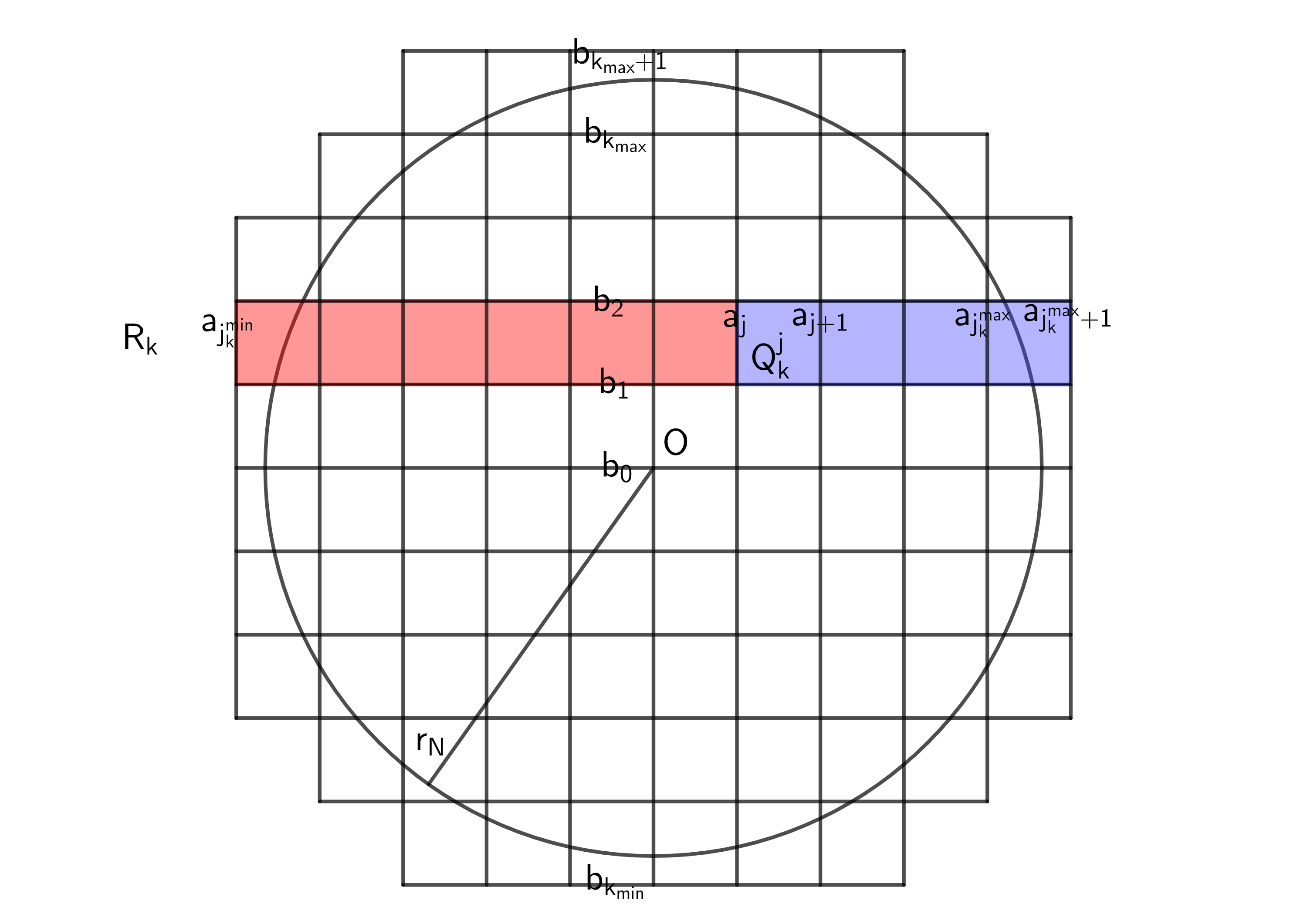}
\caption{A graphical representation of the proof. $N^j_k$ is the number of particles in $Q^j_k$ (the first square in blue), while $P^j_k$ is the number of particles in the red rectangle. Once fixed $N_k$, that is the number of particles in $R_k$, the fluctuations of the particles in the red rectangle are exactly the fluctuations of the particles in the blue one.}
\end{figure}
Now for \eqref{oksquares8} we claim that there exist  a constant $c>0$ such that $\forall j$
\begin{eqnarray*}
&&\mu(a_j,a_{j+1})e^{\frac{a_j^2}{2}}\leq c
\\
&&\frac{\mu(a_{j_k^{\min}},a_j)\mu(a_j,a_{j_k^{\max}+1})}{\mu(J^k)}e^{\frac{a_j^2}{2}}\leq c
\end{eqnarray*}
The first inequality follows from
\begin{eqnarray*}
\mu(a_j,a_{j+1})e^{\frac{a_j^2}{2}}=\int_{a_j}^{a_{j+1}}dx_1\frac{e^{-\frac{x_1^2}{2}}}{\sqrt{2\pi}}e^{\frac{a_j^2}{2}}\leq c(a_{j+1}-a_j)e^{-\frac{a_j^2}{2}}e^{\frac{a_j^2}{2}}=c(a_{j+1}-a_j)
\end{eqnarray*}
while the second one is a consequence of
\begin{eqnarray*}
\textrm { if }j\geq0&&\quad\frac{\mu(a_{j_k^{\min}},a_j)\mu(a_j,a_{j_k^{\max}+1})}{\mu(J^k)}e^{\frac{a_j^2}{2}}\leq\mu(a_j,a_{j_k^{\max}+1})e^{\frac{a_j^2}{2}}
\\
&\leq&\int_{a_j}^{+\infty}dx_1e^{-\frac{x_1^2}{2}}e^{\frac{a_j^2}{2}}\leq\max_{x\geq 0}e^{\frac{x^2}{2}}\int_x^{+\infty}dye^{-\frac{y^2}{2}}
\\
\textrm{ if }j<0&&\quad\frac{\mu(a_{j_k^{\min}},a_j)\mu(a_j,a_{j_k^{\max}+1})}{\mu(J^k)}e^{\frac{a_j^2}{2}}\leq\mu(a_{j_k^{\min}},a_j)e^{\frac{a_j^2}{2}}
\\
&\leq&\int_{-\infty}^{a_j}dx_1e^{-\frac{x_1^2}{2}}e^{\frac{a_j^2}{2}}\leq\max_{x\leq 0}e^{\frac{x^2}{2}}\int_{-\infty}^xdye^{-\frac{y^2}{2}}
\end{eqnarray*}
Thus the claim is proved and we have
\begin{eqnarray}
\nonumber&&\sum_{j=j_k^{\min}}^{j_k^{\max}}\left[\mu(a_j,a_{j+1})+\frac{\mu(a_{j_k^{\min}},a_j)\mu(a_j,a_{j_k^{\max}+1})}{\mu(J^k)}\right]e^{\frac{a_j^2}{2}}(a_{j+1}-a_j)
\\
\label{oksquares10}&\leq& c\sum_{j=j_k^{\min}}^{j_k^{\max}}(a_{j+1}-a_j)=c|J^k|\leq c\sqrt{\log N}
\end{eqnarray}
Applying \eqref{oksquares10} to \eqref{oksquares8} and using the properties of the conditioned expected value we get
\begin{eqnarray}
\nonumber&&\mathbb{E}\left[\sum_{k:N_k>0}\frac{N_k}{N}W_2^2\left(\sum_{j=j_k^{\min}}^{j_k^{\max}}\frac{N^j_k}{N_k}\frac{\rho\mathbbm{1}_{Q^j_k}}{\rho(Q^j_k)},\frac{\rho\mathbbm{1}_{R_k}}{\rho(R_k)}\right)\right]
\\
\nonumber&\leq&\mathbb{E}\left[\sum_{k:N_k>0}\frac{N_k}{N}\mathbb{E}_{N_k}\left[W_2^2\left(\sum_{j=j_k^{\min}}^{j_k^{\max}}\frac{N^j_k}{N_k}\frac{\mu\mathbbm{1}_{(a_j,a_{j+1})}}{\mu(a_j,a_{j+1})},\frac{\mu\mathbbm{1}_{J^k}}{\mu(J^k)}\right)\right]\right]
\\
\label{oksquares13}&\leq&\mathbb{E}\left[\sum_{k:N_k>0}\frac{N_k}{N}\frac{c\sqrt{\log N}}{N_k}\right]=c\frac{|{\cal R}|\sqrt{\log N}}{N}\leq c\frac{(\log N)^{\frac{3}{2}}}{\epsilon N}
\end{eqnarray}
Now we have bounded the expectation of \eqref{oksquares1}.

To estimate \eqref{oksquares2} we argue in the following way: thanks to Theorem \ref{talfor} and using $\log x\leq x-1$ and $\mathbb{E}(N_k)=N\rho_N(R_k)=N\frac{\rho(R_k)}{\rho(E_N)},$ we have
\begin{eqnarray*}
&&W_2^2\left(\sum_k\frac{N_k}{N}\frac{\rho\mathbbm{1}_{R_k}}{\rho(R_k)},\rho\right)\leq2\sum_k\frac{N_k}{N\rho(R_k)}\int_{R_k}dx\rho(x)\log\left(\frac{N_k}{N\rho(R_k)}\right)
\\
&\leq&2\sum_k\frac{\mathbb{E}(N_k)}{N}\frac{N_k}{\mathbb{E}(N_k)}\left(\frac{N_k}{\mathbb{E}(N_k)}-1\right)+2\log\left(\frac{1}{\rho(E_N)}\right)
\end{eqnarray*}
Therefore
\begin{eqnarray}
\nonumber&&\mathbb{E}\left[W_2^2\left(\sum_k\frac{N_k}{N}\frac{\rho\mathbbm{1}_{R_k}}{\rho(R_k)},\rho\right)\right]
\\
\nonumber&\leq&2\sum_k\frac{N\rho_N(R_k)(1-\rho_N(R_k))}{N^2\rho_N(R_k)}+2\log\left(\frac{1}{\rho(E_N)}\right)
\\
\label{oksquares11}&\leq&c\frac{\log N}{\epsilon N}+c\frac{(\log N)^{\alpha}}{N}
\end{eqnarray}
Before concluding the proof we can observe that for \eqref{oksquares4} at this point we would obtain the following term 
\begin{eqnarray*}
\sum_k\frac{\mathbb{E}\left(\frac{M_k-N_k}{3\theta}+M_k\right)}{N}\frac{\frac{M_k-N_k}{3\theta}+M_k}{\mathbb{E}\left(\frac{M_k-N_k}{3\theta}+M_k\right)}\left(\frac{\frac{M_k-N_k}{3\theta}+M_k}{\mathbb{E}\left(\frac{M_k-N_k}{3\theta}+M_k\right)}-1\right)
\end{eqnarray*}
that is analogue to the previous one but with a restriction to the set $A_{\theta}$: we can bound the expectation computed in $A_{\theta}$ with the expectation computed everywhere because this term is everywhere positive (indeed the function $f(x)=x^2-x$ is convex and we are considering a convex combination of the summands).

Finally, combining \eqref{oksquares1} with \eqref{oksquares13} and \eqref{oksquares2} with \eqref{oksquares11} we obtain \eqref{oksquares3}.

To prove \eqref{oksquares5} we observe that in $A_{\theta}$
\begin{eqnarray*}
&&W_2^2\left(\sum_{j,k}\frac{3\theta M^j_k}{N}\frac{\rho\mathbbm{1}_{Q^j_k}}{\rho(Q^j_k)},\sum_{j,k}\frac{M^j_k-N^j_k+3\theta M^j_k}{N}\frac{\rho\mathbbm{1}_{Q^j_k}}{\rho(Q^j_k)}\right)
\\
&\leq&6\theta W_2^2\left(\sum_{j,k}\frac{M^j_k}{N}\frac{\rho\mathbbm{1}_{Q^j_k}}{\rho(Q^j_k)},\rho\right)+6\theta W_2^2\left(\sum_{j,k}\frac{M^j_k+\frac{M^j_k-N^j_k}{3\theta}}{N}\frac{\rho\mathbbm{1}_{Q^j_k}}{\rho(Q^j_k)},\rho\right) 
\end{eqnarray*}
and this implies the thesis thanks to \eqref{oksquares3} and \eqref{oksquares4}.
\end{proof}
With the following Lemma we prove that thanks to the choice of the squares $Q^j_k$ we can transport $\frac{\rho\mathbbm{1}_{Q^j_k}}{\rho(Q^j_k)}$ in the uniform measure. This implies that the problem on the square $Q^j_k$ is (approximately) a random matching problem in the square with the uniform measure, and thus solved in \cite{AST}. 
\begin{lemma}\label{touniform}
There exist a function $\epsilon(N)\xrightarrow[N\to\infty]{}2\epsilon$ such that if $\tilde X_1,\dots,\tilde X_{N^j_k}$ and $\tilde Y_1,\dots,\tilde Y_{M^j_k}$ are independent random variables in $Q^j_k$ with common distribution $\frac{\rho_N\mathbbm{1}_{Q^j_k}}{\rho_N(Q^j_k)}$ and $Z_1,\dots,Z_{N^j_k}$ and $W_1,\dots,W_{M^j_k}$ are independent random variables in $Q^j_k$ with common distribution $\frac{\mathbbm{1}_{Q^j_k}}{|Q^j_k|}$, then
\begin{eqnarray}
\nonumber&&\mathbb{E}\left[W_2^2\left(\frac{1}{N^j_k}\sum_{i=1}^{N^j_k}\delta_{\tilde X_i},\frac{\rho_N\mathbbm{1}_{Q^j_k}}{\rho_N(Q^j_k)}\right)\right]
\\
\label{touniform1}&\leq&e^{\epsilon(N)}\mathbb{E}\left[W_2^2\left(\frac{1}{N^j_k}\sum_{i=1}^{N^j_k}\delta_{Z_i},\frac{\mathbbm{1}_{Q^j_k}}{|Q^j_k|}\right)\right]
\\
\nonumber&&\mathbb{E}\left[Wb_2^2\left(\frac{1}{N^j_k}\sum_{i=1}^{N^j_k}\delta_{\tilde X_i},\frac{1}{M^j_k}\sum_{i=1}^{M^j_k}\delta_{\tilde Y_i}\right)\right]
\\
\label{touniform2}&\geq&e^{-\epsilon(N)}\mathbb{E}\left[Wb_2^2\left(\frac{1}{N^j_k}\sum_{i=1}^{N^j_k}\delta_{Z_i},\frac{1}{M^j_k}\sum_{i=1}^{M^j_k}\delta_{W_i}\right)\right]
\end{eqnarray}
\end{lemma}

\begin{proof} To prove the statement, arguing as in \cite{BC}, we can reduce to find a suitable map $S:Q^j_k\to Q^j_k$ which transports $\frac{\rho_N\mathbbm{1}_{Q^j_k}}{\rho_N(Q^j_k)}$ in $\frac{1}{|Q^j_k|}$.
Let $S_j:(a_j,a_{j+1})\to (a_j,a_{j+1})$ and $S^k:(b_k,b_{k+1})\to(b_k,b_{k+1})$ be defined by
\begin{eqnarray*}
\int_{a_j}^{S_j(x_1)}dy_1e^{-\frac{y_1^2}{2}}=\int_{a_j}^{a_{j+1}}dy_1e^{-\frac{y_1^2}{2}}\frac{x_1-a_j}{a_{j+1}-a_j}
\\
\int_{b_k}^{S^k(x_2)}dy_2e^{-\frac{y_2^2}{2}}=\int_{b_k}^{b_{k+1}}dy_2e^{-\frac{y_2^2}{2}}\frac{x_2-b_k}{b_{k+1}-b_k}
\end{eqnarray*}
then the map $S:Q^j_k\ni x\mapsto (S_j(x_1),S^k(x_2))\in Q^j_k$ and its inverse switch the measures we are considering and fix the boundary of $Q^j_k$, and since
\begin{eqnarray*}
S_j'(x_1)&=&e^{\frac{S_j^2(x_1)}{2}}\frac{\int_{a_j}^{a_{j+1}}dy_1e^{-\frac{y_1^2}{2}}}{a_{j+1}-a_j}\in\left[e^{-\frac{|a_{j+1}^2-a_j^2|}{2}},e^{\frac{|a_{j+1}^2-a_j^2|}{2}}\right]
\\
{S^k}'(x_2)&=&e^{\frac{{S^k}^2(x_2)}{2}}\frac{\int_{b_k}^{b_{k+1}}dy_2e^{-\frac{y_2^2}{2}}}{b_{k+1}-b_k}\in\left[e^{-\frac{|b_{k+1}^2-b_k^2|}{2}},e^{\frac{|b_{k+1}^2-b_k^2|}{2}}\right]
\end{eqnarray*}
if $\epsilon(N):=2\epsilon+\frac{\epsilon^2}{r_N^2}$, we find
\begin{eqnarray*}
|a_{j+1}^2-a_j^2|\leq2\frac{|j|\epsilon^2}{r_N^2}+\frac{\epsilon^2}{r_N^2}\leq2\epsilon+\frac{\epsilon^2}{r_N^2}=\epsilon(N)
\\
|b_{k+1}^2-b_k^2|\leq2\frac{|k|\epsilon^2}{r_N^2}+\frac{\epsilon^2}{r_N^2}\leq2\epsilon+\frac{\epsilon^2}{r_N^2}=\epsilon(N)
\end{eqnarray*}
therefore 
\begin{eqnarray*}
e^{-\epsilon(N)}|x-y|^2\leq|S(x)-S(y)|^2\leq e^{\epsilon(N)}|x-y|^2
\end{eqnarray*}
and this concludes the proof.
\end{proof}

The following Lemma allows us to restrict to a good event in the bound from below for bipartite matching. It only uses Chernoff bound, as in \cite{AGT}.
\begin{lemma}\label{atheta}
Let $\tilde X_1,\dots,\tilde X_N$ and $\tilde Y_1,\dots,\tilde Y_N$ be independent random variables with common distribution $\rho_N$. Then if $\theta=\theta(N):=\frac{1}{(\log N)^{\xi}}$, $0<\xi<\frac{\alpha-1}{2}<\frac{1}{2},$ and
$$
A_{\theta}:=\bigcap_{j,k}\left\{\begin{array}{lcr}|N^j_k-N\rho_N(Q^j_k)|&\leq&\theta N\rho_N(Q^j_k)\\|M^j_k-N\rho_N(Q^j_k)|&\leq&\theta N\rho_N(Q^j_k)\end{array}\right\}
$$
it holds
$$
\mathbb{P}(A^c_{\theta})\xrightarrow[N\to\infty]{}0
$$
\end{lemma}

\begin{proof} Thanks to Chernoff bound we have
\begin{eqnarray*}
\mathbb{P}\left(|N^j_k-N\rho_N(Q^j_k)|\geq \theta N\rho_N(Q^j_k)\right)&\leq&\exp\left\{-\frac{\theta^2}{2}N\rho_N(Q^j_k)\right\}
\\
&\leq&\exp\left\{-c\epsilon^2(\log N)^{\alpha-1-2\xi}\right\}
\end{eqnarray*}
therefore
\begin{eqnarray*}
\mathbb{P}(A^c_{\theta})&\leq&2\sum_{j,k}\exp\left\{-c\epsilon^2(\log N)^{\alpha-1-2\xi}\right\}
\\
&=&2|{\cal Q}|\exp\left\{-c\epsilon^2(\log N)^{\alpha-1-2\xi}\right\}
\\
&\leq&c\frac{(\log N)^2}{\epsilon^2}\exp\left\{-c\epsilon^2(\log N)^{\alpha-1-2\xi}\right\}
\end{eqnarray*}
that implies the thesis.

\end{proof}

Finally, this last Proposition collects all the contributions to the total cost given from each square. It makes rigorous the idea explained at the beginning of this Section.
\begin{proposition}\label{totalcost}
Let $\tilde X_1,\dots,\tilde X_N$ be independent random variables with common distribution $\rho_N$. Then we have
$$
\frac{\sum_{j,k}|Q^j_k|\log\mathbb{E}(N^j_k)}{(\log N)^2}\xrightarrow[N\to\infty]{}\pi
$$
\end{proposition}

\begin{proof} First, we focus on the estimate from above, therefore we choose $0=\alpha_0<\alpha_1<\dots<\alpha_L=1$ and we define
\begin{eqnarray*}
C_l&:=&\{x\in\mathbb{R}^2|\sqrt{\alpha_l}r_N\leq|x|\leq\sqrt{\alpha_{l+1}}r_N\}
\\
W_l&:=&\bigcup_{j,k:Q^j_k\cap C_l\neq\emptyset}Q^j_k
\\
N_l&:=&\sum_{i=1}^N\mathbbm{1}(X_i\in W_l)=\sum_{j,k:Q^j_k\cap C_l\neq\emptyset}N^j_k
\end{eqnarray*}
so that
\begin{eqnarray*}
\textrm{ if $l=0$ }\quad\log\mathbb{E}(N_l)&\leq&\log N
\\
\textrm{ if $l>0$ }\quad\log\mathbb{E}(N_l)&\leq&\log\left[N\int_{\sqrt{\alpha_l}r_N-\frac{\epsilon\sqrt{2}}{r_N}}^{\sqrt{\alpha_{l+1}}r_N+\frac{\epsilon\sqrt{2}}{r_N}}dr2\pi r\frac{e^{-\frac{r^2}{2}}}{2\pi\rho(E_N)}\right]
\\
&\leq&\log\left[N\int_{\sqrt{\alpha_l}r_N-\frac{\epsilon\sqrt{2}}{r_N}}^{+\infty}dr re^{-\frac{r^2}{2}}\right]-\log\rho(E_N)
\\
&\leq&\log N(1-\alpha_l)+c\log\log N
\end{eqnarray*}
and
\begin{eqnarray*}
\sum_{j,k:Q^j_k\cap C_l\neq\emptyset}|Q^j_k|\leq2\pi(\alpha_{l+1}-\alpha_l)\log N+c
\end{eqnarray*}
that implies
\begin{eqnarray*}
\sum_{j,k}|Q^j_k|\log\mathbb{E}(N^j_k)&\leq&\sum_{l=0}^{L-1}\sum_{j,k:Q^j_k\cap C_l\neq\emptyset}|Q^j_k|\log\mathbb{E}(N^j_k)
\\
&\leq&\sum_{l=0}^{L-1}\log\mathbb{E}(N_l)\sum_{j,k:Q^j_k\cap C_l\neq\emptyset}|Q^j_k|
\\
&\leq&2\pi\sum_{l=0}^{L-1}(1-\alpha_l)(\alpha_{l+1}-\alpha_l)(\log N)^2+c\log N\log\log N
\end{eqnarray*}
therefore
\begin{eqnarray*}
\limsup_{N\to\infty}\frac{\sum_{j,k}|Q^j_k|\log\mathbb{E}(N^j_k)}{(\log N)^2}\leq2\pi\sum_{l=0}^{L-1}(1-\alpha_l)(\alpha_{l+1}-\alpha_l)
\end{eqnarray*}
Now we recognize in the right hand side of this inequality a Riemann sum for the function $f(x)=1-x$ which verifies
$$
\int_0^1dx(1-x)=\frac{1}{2}
$$
and since our choice of $\{\alpha_l\}_{l=0}^{L}$ was arbitrary in $[0,1]$ we have
\begin{eqnarray*}
\limsup_{N\to\infty}\frac{\sum_{j,k}|Q^j_k|\log\mathbb{E}(N^j_k)}{(\log N)^2}\leq\pi
\end{eqnarray*}
As for the estimate from below, since function $\log x$ is concave, for a suitable function $\zeta(N)\xrightarrow[N\to\infty]{}0$ we have
\begin{eqnarray*}
&&\sum_{j,k}|Q^j_k|\log\mathbb{E}(N^j_k)
\\
&=&\log N\sum_{j,k}|Q^j_k|+\sum_{j,k}|Q^j_k|\log\left(\frac{1}{|Q^j_k|}\int_{Q^j_k}dx\rho_N(x)\right)+\sum_{j,k}|Q^j_k|\log|Q^j_k|
\\
&\geq&2\pi(\log N)^2(1-\zeta(N))+\int_{E_N}dx\log\rho_N(x)-\log\left(\frac{2\log N}{\epsilon^2}\right)\sum_{j,k}|Q^j_k|
\\
&\geq&\pi(\log N)^2(1-\zeta(N))-c\log N\log\log N
\end{eqnarray*}
therefore
\begin{eqnarray*}
\liminf_{N\to\infty}\frac{\sum_{j,k}|Q^j_k|\log\mathbb{E}(N^j_k)}{(\log N)^2}\geq\pi
\end{eqnarray*}

\end{proof}

\subsection{Convergence Theorems}\label{subsec7}

In this Subsection we prove Theorem \ref{gauxthm}.

Thanks to Lemma \ref{semibip} it is sufficient to prove the upper bound for semidiscrete matching, in Theorem \ref{uppergaux}, and the lower bound for bipartite matching, in Theorem \ref{lowergaux}. The structure of the proof is the same as Theorem 1 in \cite{BC} and Theorem 1.1 and 1.2 in \cite{AGT}.

Both in Theorem \ref{uppergaux} and in Theorem \ref{lowergaux}, the first step consists in substituting $N$ independent random variables with common distribution $\rho$ with $N$ independent random variables with common distribution $\rho_N$. This is possible thanks to Lemma \ref{cutoff}.

Then we have to bound the distance between two measures, one of which is $\rho$ (in semidiscrete case, the first one) or the empirical measure on $N$ independent random variables with distribution $\rho_N$ (in bipartite case, the second one) while the other is the same measure as the first , multiplied in each square $Q^j_k$ by $\frac{N^j_k}{N\rho(Q^j_k)}$. This factor is expected to be very close to one, and this is the reason why the two measures involved are similar. This is possible thanks to Proposition \ref{oksquares}.

At this point, we are allowed to compute the total cost of the problem as the sum of the costs on the squares. When proving the upper bound we use a  subadditivity argument while for the lower bound we use, as in \cite{AGT}, the distance introduced in \cite{FG}, that is \eqref{figalligigli}. This distance is superadditive.

Then, since we need that in each square $Q^j_k$ there is an increasing (with $N$) number of particles, in Theorem \ref{uppergaux} we consider separately two events: the event in which the number of particles in the square is close to its expected value, and its complementary, and we show that the contribution of the second event is negligible. In Theorem \ref{lowergaux} we simply restrict to a good event (that is $A_{\theta}$) and thanks to Lemma \ref{atheta} we are sure its probability to be close to $1$.

Once made these assumptions, thanks to Lemma \ref{touniform} we can approximate the probability measure on the square $Q^j_k$ whose density is $\frac{\rho_N\mathbbm{1}_{Q^j_k}}{\rho_N(Q^j_k)}$ with the uniform measure on the square itself. Therefore, using the results obtained in \cite{AGT} and \cite{AST} except for a factor $4\pi$ or $2\pi$ the cost with uniform measure (on the square $Q^j_k$) is bounded from above and below by a term close to
$$
|Q^j_k|\frac{\log N^j_k}{N^j_k}\approx|Q^j_k|\frac{\log \mathbb{E}(N^j_k)}{N^j_k}
$$
Finally, the total cost is a convex combination of all these contributions and the main term in the estimate turns out to be
$$
\frac{\sum_{j,k}|Q^j_k|\log\mathbb{E}(N^j_k)}{N}
$$
divided by $4\pi$ in Theorem \ref{uppergaux} and $2\pi$ in Theorem \ref{lowergaux}. We have already examined it in Proposition \ref{totalcost}, and this concludes both the proofs.
\begin{theorem}\label{uppergaux}
If $X_1,\dots,X_N$ are independent random variables with common distribution $\rho$, it holds
$$
\limsup_{N\to\infty}\frac{N}{(\log N)^2}\mathbb{E}\left[W_2^2\left(\frac{1}{N}\sum_{i=1}^N\delta_{X_i},\rho\right)\right]\leq\frac{1}{4}
$$
\end{theorem}

\begin{proof} As explained before, first we substitute $X_1,\dots,X_N$ with $N$ independent random variables distributed with the probability measure whose density is $\rho_N$. If $\gamma>0$ and if $T:\mathbb{R}^2\to\mathbb{R}^2$ is the map that transports $\rho$ in $\rho_N$, we denote $\tilde X_i:=T(X_i)$. Using the triangle inequality we have
\begin{eqnarray}
\nonumber\mathbb{E}\left[W_2^2\left(\frac{1}{N}\sum_{i=1}^N\delta_{X_i},\rho\right)\right]&\leq&(1+\gamma)\mathbb{E}\left[\sum_{j,k:N^j_k>0}\frac{N^j_k}{N}W_2^2\left(\frac{1}{N^j_k}\sum_{i:\tilde X_i\in Q^j_k}\delta_{\tilde X_i},\frac{\rho\mathbbm{1}_{Q^j_k}}{\rho(Q^j_k)}\right)\right]
\\
\nonumber&+&2\frac{1+\gamma}{\gamma}\mathbb{E}\left[W_2^2\left(\frac{1}{N}\sum_{i=1}^N\delta_{X_i},\frac{1}{N}\sum_{i=1}^N\delta_{\tilde X_i}\right)\right]
\\
\label{uppergaux1}&+&2\frac{1+\gamma}{\gamma}\mathbb{E}\left[W_2^2\left(\sum_{j,k}\frac{N^j_k}{N}\frac{\rho\mathbbm{1}_{Q^j_k}}{\rho(Q^j_k)},\rho\right)\right]
\end{eqnarray}
The first term in the sum in \eqref{uppergaux1} gives the main contribution, while we have a bound for the second and the third one thanks to \eqref{cutoff2} of Lemma \ref{cutoff} and \eqref{oksquares3} of Proposition \ref{oksquares}. So we only focus on the first term.

To estimate it first we exclude the events with few particles in any square $Q^j_k$, therefore we define
\begin{eqnarray*}
A^j_k:=\left\{N^j_k\geq\frac{\mathbb{E}(N^j_k)}{2}\right\}=\left\{N^j_k\geq\frac{N\rho_N(Q^j_k)}{2}\right\}
\end{eqnarray*}
and we observe that the contributions in the events ${A^j_k}^c$ are negligible, indeed
\begin{eqnarray*}
&&\mathbb{E}\left[\frac{N^j_k}{N}W_2^2\left(\frac{1}{N^j_k}\sum_{i:\tilde X_i\in Q^j_k}\delta_{\tilde X_i},\frac{\rho\mathbbm{1}_{Q^j_k}}{\rho(Q^j_k)}\right)\mathbbm{1}_{{A^j_k}^c}\mathbbm{1}(N^j_k>0)\right]
\\
&\leq&\frac{\rho_N(Q^j_k)}{2}\frac{2\epsilon^2}{r_N^2}\mathbb{P}({A^j_k}^c)\leq\frac{\rho_N(Q^j_k)}{2}\frac{2\epsilon^2}{r_N^2}\frac{4}{N\rho_N(Q^j_k)}=\frac{4|Q^j_k|}{N}
\end{eqnarray*}
and therefore
\begin{eqnarray}
\nonumber&&\mathbb{E}\left[\sum_{j,k:N^j_k>0}\frac{N^j_k}{N}W_2^2\left(\frac{1}{N^j_k}\sum_{i:\tilde X_i\in Q^j_k}\delta_{\tilde X_i},\frac{\rho\mathbbm{1}_{Q^j_k}}{\rho(Q^j_k)}\right)\right]
\\
\nonumber&\leq&\mathbb{E}\left[\sum_{j,k}\frac{N^j_k}{N}W_2^2\left(\frac{1}{N^j_k}\sum_{i:\tilde X_i\in Q^j_k}\delta_{\tilde X_i},\frac{\rho\mathbbm{1}_{Q^j_k}}{\rho(Q^j_k)}\right)\mathbbm{1}_{A^j_k}\right]+4\frac{\sum_{j,k}|Q^j_k|}{N}
\\
\label{uppergaux2}&\leq&\sum_{j,k}\mathbb{E}\left[\frac{N^j_k}{N}\mathbbm{1}_{A^j_k}\mathbb{E}_{N^j_k}\left[W_2^2\left(\sum_{i:\tilde X_i\in Q^j_k}\delta_{\tilde X_i},\frac{\rho\mathbbm{1}_{Q^j_k}}{\rho(Q^j_k)}\right)\right]\right]+c\frac{\log N}{N}
\end{eqnarray}
Then, we use \eqref{touniform1} of Lemma \ref{touniform}, therefore if $Z_i:=S(\tilde X_i)$ where $S$ is the map that transports $\frac{\rho\mathbbm{1}_{Q^j_k}}{\rho(Q^j_k)}$ in the uniform measure on the square $Q^j_k$, we have
\begin{eqnarray*}
\mathbb{E}_{N^j_k}\left[W_2^2\left(\frac{1}{N^j_k}\sum_{i:X_i\in Q^j_k}\delta_{X_i},\frac{\rho\mathbbm{1}_{Q^j_k}}{\rho(Q^j_k)}\right)\right]\leq e^{\epsilon(N)}\mathbb{E}_{N^j_k}\left[W_2^2\left(\frac{1}{N^j_k}\sum_{i:\tilde X_i\in Q^j_k}\delta_{Z_i},\frac{\mathbbm{1}_{Q^j_k}}{|Q^j_k|}\right)\right]
\end{eqnarray*}
for a suitable function $\epsilon(N)\xrightarrow[N\to\infty]{}2\epsilon$.

Moreover, since in the event $A^j_k$, $N^j_k\geq\frac{\mathbb{E}(N^j_k)}{2}\geq\frac{\epsilon^2}{r_N^2}\frac{e^{-\frac{r_N^2}{2}}}{2\sqrt{2\pi}}\geq c\epsilon^2(\log N)^{\alpha-1}$, we can use \eqref{teoremone1} of Theorem \ref{teoremone} to find a function $\omega(N)\xrightarrow[N\to\infty]{}0$ such that in $A^j_k$
\begin{eqnarray*}
\mathbb{E}_{N^j_k}\left[W_2^2\left(\frac{1}{N^j_k}\sum_{i:\tilde X_i\in Q^j_k}\delta_{\tilde X_i},\frac{\rho\mathbbm{1}_{Q^j_k}}{\rho(Q^j_k)}\right)\right]\leq e^{\epsilon(N)}|Q^j_k|\frac{\log N^j_k}{4\pi N^j_k}\left(1+\omega(N)\right)
\end{eqnarray*}
which leads to
\begin{eqnarray}
\nonumber&&\mathbb{E}\left[\frac{N^j_k}{N}\mathbbm{1}_{A_j}\mathbb{E}_{N^j_k}\left[W_2^2\left(\sum_{i:\tilde X_i\in Q^j_k}\delta_{\tilde X_i},\frac{\rho\mathbbm{1}_{Q^j_k}}{\rho(Q^j_k)}\right)\right]\right]
\\
\nonumber&\leq&\mathbb{E}\left[\frac{N^j_k}{N}\mathbbm{1}_{A_j}e^{\epsilon(N)}|Q^j_k|\frac{\log N^j_k}{4\pi N^j_k}\left[1+\omega(N)\right]\right]
\\
\label{uppergaux3}&\leq&|Q^j_k|\frac{\log \mathbb{E}(N^j_k)}{4\pi N}\left[1+\omega(N)\right]
\end{eqnarray}
Finally, combining \eqref{uppergaux1}, \eqref{uppergaux2} and \eqref{uppergaux3} and using Lemma \ref{cutoff} and Proposition \ref{oksquares}, we have
\begin{eqnarray*}
\mathbb{E}\left[W_2^2\left(\frac{1}{N}\sum_{i=1}^N\delta_{X_i},\rho\right)\right]\leq(1+\gamma)e^{\epsilon(N)}\frac{\sum_{j,k}|Q^j_k|\log\mathbb{E}(N^j_k)}{4\pi N}\left[1+\omega(N)\right]
\\
+(1+\gamma)c\frac{\log N}{N}+2\frac{1+\gamma}{\gamma}c\frac{(\log N)^{\alpha}}{N}+2\frac{1+\gamma}{\gamma}c\left[\frac{(\log N)^{\frac{3}{2}}}{\epsilon N}+\frac{(\log N)^{\alpha}}{N}\right]
\end{eqnarray*}
Using Proposition \ref{totalcost} we find
\begin{eqnarray*}
\limsup_{N\to\infty}\frac{N}{(\log N)^2}\mathbb{E}\left[W_2^2\left(\frac{1}{N}\sum_{i=1}^N\delta_{X_i},\rho\right)\right]\leq\frac{(1+\gamma)e^{2\epsilon}}{4}
\end{eqnarray*}
and letting $\gamma,\epsilon\to0$ we conclude.

\end{proof}

\begin{theorem}\label{lowergaux}
If $X_1,\dots,X_N$ and $Y_1,\dots,Y_N$ are independent random variables with common distribution $\rho$, we have
\begin{eqnarray*}
\liminf_{N\to\infty}\frac{N}{(\log N)^2}\mathbb{E}\left[W_2^2\left(\frac{1}{N}\sum_{i=1}^N\delta_{X_i},\frac{1}{N}\sum_{i=1}^N\delta_{Y_i}\right)\right]\geq\frac{1}{2}
\end{eqnarray*}
\end{theorem}

\begin{proof} \textit{(Sketch)} Except for some steps, the proof is very similar to the previous one, therefore we only underline the differences. Once made the substitution of $\{X_i,Y_i\}_{i=1}^N$ with $\{\tilde X_i,\tilde Y_i\}_{i=1}^N$, we restrict to the set
\begin{eqnarray*}
A_{\theta}:=\bigcap_{j,k}\left\{\begin{array}{lcr}|N^j_k-N\rho_N(Q^j_k)|&\leq&\theta N\rho_N(Q^j_k)\\|M^j_k-N\rho_N(Q^j_k)|&\leq&\theta N\rho_N(Q^j_k)\end{array}\right\}
\end{eqnarray*}
where $\theta=\frac{1}{(\log N)^{\xi}}$ and $0<\xi<\frac{\alpha-1}{2}<\frac{1}{2}$.

Then, if we rename
\begin{eqnarray*}
\mu^{N^j_k}:=\frac{1}{N^j_k}\sum_{i:\tilde X_i\in Q^j_k}\delta_{\tilde X_i}&;&\nu^{M^j_k}:=\frac{1}{M^j_k}\sum_{i:\tilde Y_i\in Q^j_k}\delta_{\tilde Y_i}
\end{eqnarray*}
we apply triangle inequality and superadditivity of $Wb_2^2$ to obtain
\begin{eqnarray}
\nonumber&&\mathbb{E}\left[W_2^2\left(\frac{1}{N}\sum_{i=1}^N\delta_{\tilde X_i},\frac{1}{N}\sum_{i=1}^N\delta_{\tilde Y_i}\right)\mathbbm{1}_{A_{\theta}}\right]
\\
\label{lowergaux1}&\geq&(1-\delta)\mathbb{E}\left[\sum_{j,k}\frac{N^j_k}{N}Wb_2^2(\mu^{N^j_k},\nu^{M^j_k})\mathbbm{1}_{A_{\theta}}\right]
\\
\label{lowergaux2}&-&\frac{1-\delta}{\delta}\mathbb{E}\left[W_2^2\left(\sum_{j,k}\frac{N^j_k}{N}\nu^{M^j_k},\sum_{j,k}\frac{M^j_k}{N}\nu^{M^j_k}\right)\mathbbm{1}_{A_{\theta}}\right]
\end{eqnarray}
The main term is \eqref{lowergaux1} and it can be estimated with \eqref{touniform2} and \eqref{teoremone2} of Theorem \ref{teoremone}, as in the previous Theorem.

Then, as proved in Lemma \ref{forbipartitematching} and Proposition \ref{oksquares}, up for a constant, the term in \eqref{lowergaux2} is bounded by
\begin{eqnarray*}
\frac{(\log N)^{2-\xi}}{N}+\frac{(\log N)^{\frac{3}{2}+\xi}}{\epsilon N}+c\frac{(\log N)^{\alpha-\xi}}{N}
\end{eqnarray*}
and, sending $\delta$ and $\epsilon$ (which is hidden in \eqref{lowergaux1}) to 0 we get the proof.

\end{proof}

\section{The Maxwellian density}\label{sec1}

In this Section we briefly focus on the Maxewellian density, that is
$$
\rho(x):=\mathbbm{1}_{[0,1]}(x_1)\mu(x_2)\quad;\quad\mu(z):=\frac{e^{-\frac{z^2}{2}}}{\sqrt{2\pi}}
$$
$\rho$ is uniform in one direction and Gaussian in the other one.

We consider $X_1,\dots,X_N$ and $Y_1,\dots,Y_N$ independent random variables with values in $(0,1)\times\mathbb{R}$ with density $\rho$.

Arguing as in the case of the Gaussian density, we can prove the following result
\begin{theorem}\label{Maxthm}
If $X_1,\dots,X_N$ and $Y_1,\dots,Y_N$ are independent random variables in $\mathbb{R}^2$ with common distribution $\rho$, we have
\begin{eqnarray*}
&&\frac{N}{(\log N)^{\frac{3}{2}}}\mathbb{E}\left[W_2^2\left(\frac{1}{N}\sum_{i=1}^N\delta_{X_i},\rho\right)\right]\xrightarrow[N\to\infty]{}\frac{\sqrt{2}}{3\pi}
\\
&&\frac{N}{(\log N)^{\frac{3}{2}}}\mathbb{E}\left[W_2^2\left(\frac{1}{N}\sum_{i=1}^N\delta_{X_i},\frac{1}{N}\sum_{i=1}^N\delta_{Y_i}\right)\right]\xrightarrow[N\to\infty]{}\frac{2\sqrt{2}}{3\pi}
\end{eqnarray*}
\end{theorem}
The idea is again that the number of particles $X_i$ or $Y_i$ close to a point $x\in(0,1)\times\mathbb{R}$ is $Ne^{-\frac{x_2^2}{2}}dx$, that is strictly smaller than 1 when $|x_2|>\sqrt{2\log N}$. Therefore, also if the density $\rho$ has an unbounded support, we expect all the particles to be in the domain $(0,1)\times(-\sqrt{2\log N},\sqrt{2\log N})$. As for the case of the Gaussian density we expect the total cost to be estimated by
$$
\frac{1}{N}\int_0^1dx_1\int_{-\sqrt{2\log N}}^{\sqrt{2\log N}}dx_2\log\left(N\frac{e^{-\frac{x_2^2}{2}}}{\sqrt{2\pi}}\right)=\frac{4\sqrt{2}}{3}\frac{(\log N)^{\frac{3}{2}}}{N}+\frac{{\cal O}(1)}{N}
$$
once omitted the factor $\frac{1}{2\pi}$ or $\frac{1}{4\pi}$.

To make this rigorous we substitute again $\rho$ with a density (that we will call $\rho_N$) whose support is compact but increases with $N$. To define $\rho_N$, first we define $r_N$ and $\tilde r_N$ as
$$
r_N:=\sqrt{2\log\left(\frac{N}{(\log N)^{\alpha}}\right)},\qquad1<\alpha<2,\qquad\tilde r_N:=\frac{\lfloor m\lfloor r_N\rfloor r_N\rfloor+1}{m\lfloor r_N\rfloor}
$$
We define $\tilde r_N$ in this way because when multiplying for $m\lfloor r_N\rfloor$ we obtain an integer number and therefore the set can easily covered by an integer number of squares, while $\frac{\tilde r_N}{r_N}\xrightarrow[N\to\infty]{}1$. Moreover, $m$ will be sent to $\infty$ in the end.

Then we apply the cut-off:
\begin{eqnarray*}
\rho_N:=\frac{\rho\mathbbm{1}_{(0,1)\times(-\tilde r_N,\tilde r_N)}}{\rho((0,1)\times(-\tilde r_N,\tilde r_N))}
\end{eqnarray*}
Finally we cover $(0,1)\times(-\tilde r_N,\tilde r_N)$ with squares $\{Q^j_k\}_{j,k}$ whose side is $\frac{1}{m\lfloor r_N\rfloor}$, and we also define the horizontal rectangles $\{R_k\}_{k}:=\{\cup_j Q^j_k\}_k$, which are aimed to be used in the analogous of the Proposition \ref{oksquares}.
\begin{figure}[h!]
\centering
\includegraphics[width=6cm]{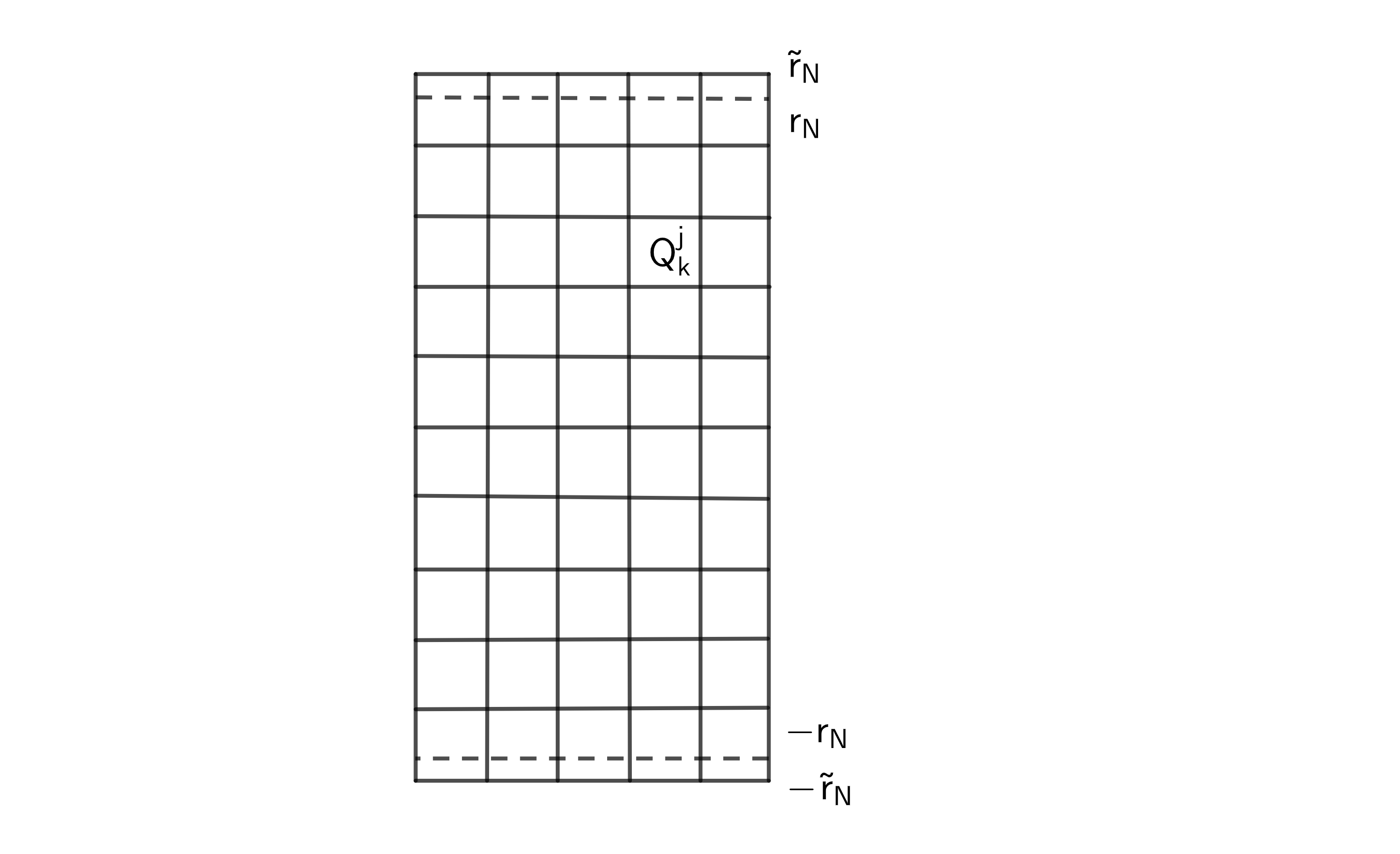}
\includegraphics[width=6cm]{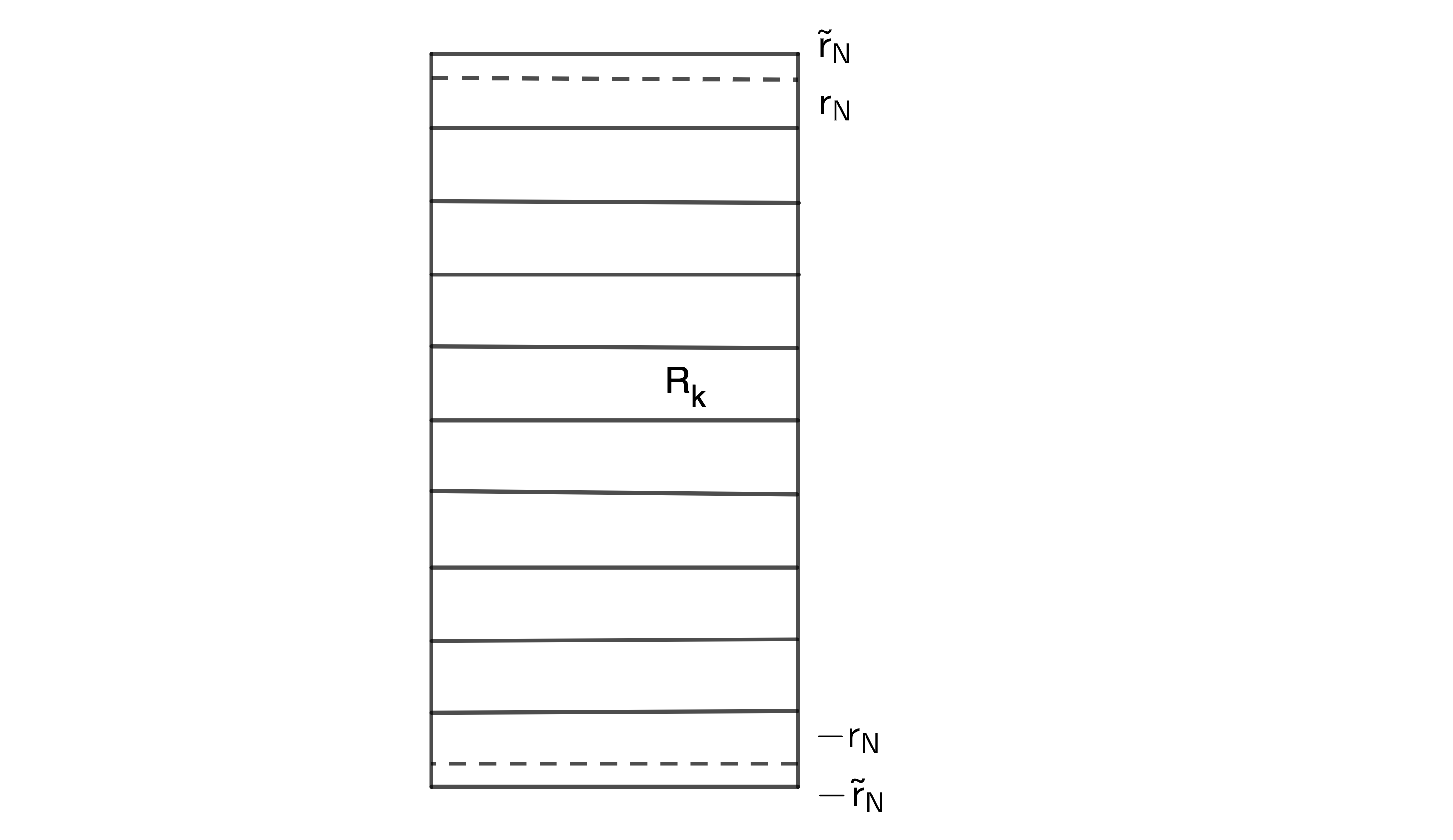}
\caption{On the left: the squares $\{Q^j_k\}_{j,k}$. On the right: the horizontal rectangles $\{R_k\}_k$.}
\end{figure}

Using this partition the arguments used are analogous to the Gaussian case, and in this way we can prove Theorem \ref{Maxthm}.

\section{Appendix}

\begin{lemma}\label{productproperty}
Let $\mu$ and $\lambda$ be two probability measures on $\mathbb{R}$ absolutely continuous with respect to Lebesgue measure, and let $\nu$ be any probability measure on $\mathbb{R}$. Then
$$
W_2^2\left(\mu\otimes\nu,\lambda\otimes\nu\right)\leq W_2^2\left(\mu,\lambda\right)
$$
\end{lemma}

\begin{proof}  If $S:\mathbb{R}\to\mathbb{R}$ is the optimal map that transports $\mu$ in $\lambda$, i.e.
\begin{eqnarray}
\label{productproperty1}\lambda(A)&=&\mu(S^{-1}(A))
\\
\label{productproperty2}W_2^2(\mu,\lambda)&=&\int_{\mathbb{R}}dz\mu(z)|z-S(z)|^2
\end{eqnarray}
the map $T:\mathbb{R}\times\mathbb{R}\ni(x_1,x_2)\mapsto(S(x_1),x_2)\in\mathbb{R}\times\mathbb{R}$ transports $\mu\otimes\nu$ in $\lambda\otimes\nu$, indeed using \eqref{productproperty1} we have
$$
\lambda\otimes\nu(A\times B)=\lambda(A)\nu(B)=\mu(S^{-1}(A))\nu(B)=\mu\otimes\nu(T^{-1}(A\times B))
$$
therefore, using \eqref{productproperty2}
\begin{eqnarray*}
W_2^2\left(\mu\otimes\nu,\lambda\otimes\nu\right)&\leq&\int_{\mathbb{R}^2}d\mu\otimes\nu(x)|x-T(x)|^2
\\
&=&\int_{\mathbb{R}}dx_1\mu(x_1)|x_1-S(x_1)|^2=W_2^2(\mu,\lambda)
\end{eqnarray*}

\end{proof}

\begin{lemma}[\cite{AGT}, Proof of Theorem 1.2, \textit{Step 3}]\label{forbipartitematching}
Let $\rho$ (also depending on $N$) be any of the probability densities we are considering and $\Omega\subseteq\mathbb{R}^2$ its support ($\Omega$ depending on $N$ too). Let $\{X_i\}_{i=1}^N$ and $\{Y_i\}_{i=1}^N$ be iid random variables with distribution $\rho$.

Let $\{\Omega_h\}_h$ be the partition of $\Omega$ we have considered. We denote by $N_h$ and $M_h$ respectively the number of $\{X_i\}_{i=1}^N$ and $\{Y_i\}_{i=1}^N$ in $\Omega_h$ and
$$
\mu^{N_h}:=\frac{1}{N_h}\sum_{i:X_i\in\Omega_h}\delta_{X_i}\qquad;\qquad\nu^{M_h}:=\frac{1}{M_h}\sum_{i:Y_i\in\Omega_h}\delta_{Y_i}
$$
For $\theta\in(0,1)$ let $A_{\theta}$ be the set
\begin{eqnarray*}
A_{\theta}:=\bigcap_h\left\{\begin{array}{lcr}|N_h-\mathbb{E}(N_h)|&\leq&\theta\mathbb{E}(N_h)\\|M_h-\mathbb{E}(M_h)|&\leq&\theta\mathbb{E}(M_h)\end{array}\right\}
\end{eqnarray*}
Then
\begin{eqnarray*}
&&\mathbb{E}\left[W_2^2\left(\sum_h\frac{N_h}{N}\nu^{M_h},\sum_k\frac{M_h}{M}\nu^{M_h}\right)\mathbbm{1}_{A_{\theta}}\right]
\\
&\leq&c\theta\frac{\sum_h|\Omega_h|\log\mathbb{E}(N_h)}{N}
\\
&+&\mathbb{E}\left[W_2^2\left(\sum_h\frac{3\theta M_h}{N}\frac{\rho\mathbbm{1}_{\Omega_h}}{\rho(\Omega_h)},\frac{M_h-N_h+3\theta M_h}{N}\frac{\rho\mathbbm{1}_{\Omega_h}}{\rho(\Omega_h)}\right)\mathbbm{1}_{A_{\theta}}\right]
\end{eqnarray*}
\end{lemma}

\newpage

\section*{Data availability statement} Data sharing not applicable to this article as no datasets were generated or analysed during the current study.
\section*{Conflict of interest}
The authors have no competing interests to declare that are relevant to the content of this article.

\end{document}